\documentclass[12pt, reqno]{amsart}

\usepackage{amsmath,amssymb,amsthm,amsfonts,verbatim}
\usepackage{microtype}
\usepackage[all,2cell]{xy}
\usepackage{mathtools}
\usepackage{graphicx}
\usepackage{pinlabel}
\usepackage{hyperref}
\usepackage{mathrsfs}
\usepackage{color}
\usepackage[dvipsnames]{xcolor}
\usepackage{enumerate}
\usepackage{soul}

\CompileMatrices

\usepackage[top=1.2in,bottom=1.2in,left=1in,right=1in]{geometry}

\usepackage{hyperref} 
\hypersetup{          
	colorlinks=true, breaklinks, linkcolor=[RGB]{51 102 204}, filecolor=Orchid, urlcolor=[RGB]{51 102 204},
	citecolor=Orchid, linktoc=all, }
\usepackage[nameinlink]{cleveref}

\theoremstyle{plain}
\newtheorem{theorem}{Theorem}[section]
\newtheorem{maintheorem}{Theorem}

\newtheorem{maincor}[maintheorem]{Corollary}

\newtheorem{proposition}[theorem]{Proposition}
\newtheorem{lemma}[theorem]{Lemma}
\newtheorem{conjecture}[theorem]{Conjecture}

\theoremstyle{definition}
\newtheorem{definition}[theorem]{Definition}
\newtheorem{example}[theorem]{Example}
\newtheorem{construction}[theorem]{Construction}

\newtheorem{remark}[theorem]{Remark}

\newcommand{\nc}{\newcommand}
\nc{\dmo}{\DeclareMathOperator}

\nc{\cA}{\mathcal{A}}
\nc{\sB}{\mathscr{B}}
\nc{\C}{\mathbb{C}}
\nc{\cD}{\mathcal{D}}
\nc{\bF}{\mathbb{F}}
\nc{\cF}{\mathcal{F}}
\nc{\cI}{\mathcal{I}}
\nc{\cK}{\mathcal{K}}
\nc{\cEll}{\mathcal{L}}
\nc{\cM}{\mathcal{M}}
\nc{\bM}{\mathbf{M}}
\nc{\N}{\mathbb{N}}
\nc{\Q}{\mathbb{Q}}
\nc{\R}{\mathbb{R}}
\nc{\cS}{\mathcal{S}}
\nc{\cT}{\mathcal T}
\nc{\cU}{\mathcal U}
\nc{\Z}{\mathbb{Z}}
\nc{\disk}{\mathbb{D}}
\nc{\hyp}{\mathbb{H}}

\nc{\CP}{\mathbb{CP}}
\nc{\RP}{\mathbb{RP}}
\dmo{\Mod}{Mod}
\dmo{\PMod}{PMod}
\dmo{\LMod}{LMod}
\dmo{\Diff}{Diff}
\dmo{\Homeo}{Homeo}
\dmo{\dist}{dist}
\dmo\BDiff{BDiff}
\dmo\SO{SO}
\dmo\Hom{Hom}
\dmo\SL{SL}
\dmo\rank{rank}
\dmo\sig{sig}
\dmo\Out{Out}
\dmo\Aut{Aut}
\dmo\Inn{Inn}
\dmo\GL{GL}
\dmo\PGL{PGL}
\dmo\Gr{Gr}
\dmo\PSL{PSL}
\dmo\BHomeo{BHomeo}
\dmo\EHomeo{EHomeo}
\dmo\EDiff{EDiff}
\dmo\Disc{Disc}
\dmo\Aff{Aff}

\renewcommand{\bar}{\overline}

\dmo\Teich{Teich}
\dmo\Fix{Fix}
\nc{\pair}[1]{\ensuremath{\left\langle #1 \right\rangle}}
\nc{\abs}[1]{\ensuremath{\left| #1 \right|}}
\nc{\action}{\circlearrowright}
\nc{\norm}[1]{\ensuremath{\left | \left | #1 \right | \right |}}
\nc{\abcd}[4]{\ensuremath{\left(\begin{array}{cc} #1 & #2 \\ #3 & #4 \end{array}\right)}}
\nc{\into}\hookrightarrow
\dmo{\Isom}{Isom}
\nc{\normal}{\vartriangleleft}
\dmo{\Vol}{Vol}
\dmo{\im}{Im}
\dmo{\Push}{Push}
\dmo{\Conf}{Conf}
\dmo{\UConf}{UConf}
\dmo{\PConf}{PConf}
\dmo{\Poly}{Poly}
\dmo{\PB}{PB}
\dmo{\id}{id}
\dmo{\Jac}{Jac}
\dmo{\Pic}{Pic}
\dmo{\Stab}{Stab}
\dmo{\Arf}{Arf}
\dmo{\End}{End}
\dmo{\Gal}{Gal}
\dmo{\lcm}{lcm}
\dmo{\ab}{ab}
\dmo{\opp}{op}
\dmo{\SU}{SU}
\dmo{\OT}{\Omega \mathcal{T}}
\dmo{\OM}{\Omega \mathcal{M}}
\dmo{\PH}{\mathbb{P}\mathcal{H}}
\dmo{\spin}{spin}
\dmo{\even}{even}
\dmo{\odd}{odd}
\dmo{\comp}{\mathcal{H}}
\dmo{\Mgk}{\mathcal{M}_{g, \underline{\kappa}}}
\dmo{\orb}{orb}
\dmo{\AJ}{AJ}
\dmo{\Ck}{\mathsf{C}(\underline{\kappa})}
\dmo{\Int}{Int}
\dmo{\pr}{pr}
\dmo{\lab}{lab}
\dmo{\Sym}{Sym}
\dmo{\Ann}{Ann}
\dmo{\Rad}{Rad}
\dmo{\Ind}{Ind}
\dmo{\Div}{Div}
\dmo{\Res}{Res}
\dmo{\Hur}{Hur}
\dmo{\vcd}{vcd}

\nc{\Span}[1]{\operatorname{Span}(#1)}

\renewcommand{\epsilon}{\varepsilon}

\renewcommand{\le}{\leqslant}
\nc{\coloneq}{\mathrel{\mathop:}\mkern-1.2mu=}
\nc{\margin}[1]{\marginpar{\scriptsize #1}}
\nc{\para}[1]{\medskip\noindent\textbf{#1.}}
\definecolor{myblue}{RGB}{102,153, 255}
\definecolor{myred}{RGB}{204,0,0}
\definecolor{mygreen}{RGB}{0,204,0}
\definecolor{myorange}{RGB}{255,102,0}
\definecolor{mypurple}{RGB}{138,43,226}
\definecolor{myyellow}{RGB}{255,204,0}
\nc{\red}[1]{\textcolor{myred}{#1}}
\nc{\blue}[1]{\textcolor{myblue}{#1}}

\nc{\Adm}{\mathbf{Adm}}
\nc{\Env}{\mathbf{Env}}
\dmo{\ord}{ord}
\nc{\MD}{\mathcal{MD}}

\author{Nick Salter}

\address{Nick Salter: Department of Mathematics, University of Notre Dame, 255 Hurley Building, Notre Dame, IN 46556}
\email{nsalter@nd.edu}

\date{August 12, 2024}

\title{Monodromy of stratified braid groups, II}

\begin{document}
\maketitle

\begin{abstract}
The space of monic squarefree polynomials has a stratification according to the multiplicities of the critical points, called the equicritical stratification. Tracking the positions of roots and critical points, there is a map from the fundamental group of a stratum into a braid group. We give a complete determination of this map. It turns out to be characterized by the geometry of the translation surface structure on $\mathbb{CP}^1$ induced by the logarithmic derivative $df/f$ of a polynomial in the stratum.
\end{abstract}

\section{Introduction}
Let $\Poly_n(\C)$ denote the space of monic squarefree complex polynomials of degree $n$. Associating a polynomial to its roots and vice versa, this can equivalently be described as the space $\UConf_n(\C)$ of unordered $n$-tuples of distinct points in $\C$; its fundamental group is the braid group $B_n$ on $n$ strands. The focus of this paper is on the {\em equicritical stratification} $\{\Poly_n(\C)[\kappa]\}$ on $\Poly_n(\C)$, previously introduced in \cite{stratbraid1}. Here, $\kappa = k_1 \ge \dots \ge k_p$ is a partition of $n-1$, and a polynomial $f \in \Poly_n(\C)$ belongs to $\Poly_n(\C)[\kappa]$ if and only if the critical points of $f$ (i.e. roots of $f'$) have multiplicities specified by $\kappa$. 

One of the most fundamental problems about $\Poly_n(\C)[\kappa]$ is to understand its fundamental group, a so-called {\em stratified braid group}
\[
\sB_n[\kappa]:= \pi_1(\Poly_n(\C)[\kappa]).
\]
One would certainly expect this to be closely related to the braid group. Indeed, for a partition $\kappa$ of $n-1$ with $p$ parts, the stratum $\Poly_n(\C)[\kappa]$ admits an embedding into the configuration space $\UConf_{n+p}(\C)$, by associating $f \in \Poly_n(\C)[\kappa]$ to the $(n+p)$-tuple of its roots and critical points. Taking $\pi_1$, we obtain a {\em monodromy map} $\rho: \sB_n[\kappa] \to B_{n+p}$. 

Our main result gives a complete description of the image of $\rho$. We find that it is characterized by a structure known as a {\em relative winding number function}, as defined in \Cref{section:framedbraid}. The subgroup of the braid group preserving a given winding number function is called a {\em framed braid group}\footnote{This terminology is chosen to mirror the ``framed mapping class groups'' studied in \cite{strata3}. This should not be confused with a braid group relative to fixed tangent vectors at the marked points, which have also been given this name.} - see \Cref{subsection:framedbraid}. A particular ``logarithmic'' relative winding number function $\psi_T$ arises in our setting by considering the geometry of $\CP^1$ equipped with the logarithmic derivative $df/f$ - see \Cref{section:logdiff}.

\begin{maintheorem}\label{theorem:main}
    For all $n \ge 3$ and all partitions $\kappa = k_1 \ge \dots \ge k_p$ of $n-1$ with $p \ge 2$ parts, the monodromy map
    \[
    \rho: \sB_n[\kappa] \to B_{n+p}
    \]
    has image $B_\kappa[\psi_T]$, the framed braid group associated to the logarithmic relative winding number function $\psi_T$. Since $\sB_n[\kappa]$ is finitely generated (being the fundamental group of a quasiprojective variety), the same is true of the framed braid group $B_\kappa[\psi_T]$.
\end{maintheorem}

\begin{remark}
    \Cref{theorem:main} does not apply in the case $p = 1$ of a single critical point, but it is easy to get a complete understanding of what happens in this case. Necessarily $\kappa = \{n-1\}$, and it is readily seen that any $f \in \Poly_n(\C)[n-1]$ is of the form $f(z) = (z-\alpha)^n + \beta$ for $\alpha \in \C$ and $\beta \in \C^*$. Thus $\Poly_n(\C)[n-1]$ can be identified with $\C^* \ltimes \C$ (indeed, it carries a free and transitive action, and hence is a {\em torsor} for $\C^* \ltimes \C$). In particular, it has cyclic fundamental group, and the monodromy image is seen to be generated by a ``$1/n$ rotation'', arranging the $n$ roots at roots of unity and applying a rotation by $2 \pi /n$. 
\end{remark}

We hope to use \Cref{theorem:main} as a stepping-stone to a complete determination of $\sB_n[\kappa]$.
\begin{conjecture}\label{conjecture:iso}
For $\kappa = k_1 \ge \dots \ge k_p$ with $p \ge 3$ parts, $\rho$ is injective, and hence there is an isomorphism
\[
\sB_n[\kappa] \cong B_\kappa[\psi_T].
\]
\end{conjecture}

It is necessary to include the hypothesis $p \ge 3$; indeed, P. Huxford and the author have shown (in yet-unpublished work) that $\rho$ is {\em never} injective for $p = 2$. This appears to be a low-complexity phenomenon arising from the close connection between $\sB_n[\kappa]$ and free groups, which is unique to the case $p = 2$. For $p \ge 3$, new relations arise in $\sB_n[\kappa]$ that we believe are sufficient to ensure injectivity of $\rho$. We plan to return to this question in future work; the complex of ``admissible root markings'' studied in \Cref{section:connectivity} of this paper will be essential to our approach.

In \cite[Theorem A]{stratbraid1} we obtained a weaker version of \Cref{theorem:main}, in which we considered only the braiding of the roots, ignoring the critical points. We found there that the image is similarly governed by a weak analogue of a relative winding number function; the analogous subgroup of the braid group is called an {\em $r$-spin braid group}. The methods of proof are almost completely different, and notably, the version in \cite{stratbraid1} only applied in a range that excluded certain cases. As a corollary, we are able to strengthen the result of \cite[Theorem A]{stratbraid1}, showing that it holds in the maximal possible range.

\begin{maincor}\label{corollary:main}
    For all $n \ge 3$ and all ordered partitions $\kappa = k_1 \ge \dots \ge k_p$ of $n-1$ with $p \ge 2$ parts, the root monodromy map
    \[
    \bar \rho: \sB_n[\kappa] \to B_n
    \]
    has image 
    \[
    \bar \rho (\sB_n[\kappa]) = B_n[\bar \psi_T],
    \]
    an $r$-spin braid group.
\end{maincor}

\para{Context: strata of differentials} As discussed in \cite{stratbraid1} and used throughout below, an equicritical stratum $\Poly_n(\C)[\kappa]$ is closely related to a particular stratum of meromorphic differentials on $\CP^1$, by associating $f$ to its logarithmic derivative $df/f$. The study of equicritical strata therefore fits into the larger enterprise of understanding the topology of strata of meromorphic and abelian differentials, as pioneered by Kontsevich--Zorich \cite{KZ}.

Already in \cite{KZpre}, the question was raised of determining the fundamental groups of strata, originally in the setting of holomorphic differentials on Riemann surfaces of higher genus. This question has remained stubbornly resistant to attack, apart from Kontsevich--Zorich's work on the hyperelliptic case in \cite{KZ}, as well as the beautiful work of Looijenga--Mondello \cite{LM} which describes the (orbifold) fundamental groups of many strata of differentials in genus $3$. 

Our interest in the equicritical stratum is motivated in large part by our belief that it should serve as a useful test case for the more general study of topological aspects of strata: it is rich enough to require the development of new methods, while remaining tractable enough to actually be amenable to study. 

We should also mention the close connection between the equicritical stratification and the study of the ``isoresidual fibration'' as appearing in the work of Gendron--Tahar \cite{gt, gt2}. There, the interest is in studying the space of all meromorphic differentials on $\CP^1$ with fixed orders of zeroes and poles. The space of polynomial logarithmic derivatives arises here as a fiber of the {\em isoresidual map} assigning such a differential to its vector of residues - by the argument principle, the residue at each zero of $f$ is $2 \pi i$. 

\para{Context: configuration spaces as spaces of polynomials} The results of this paper also fit into the literature on the study of the braid group by way of the isomorphism $\UConf_n(\C) \cong \Poly_n(\C)$. Prior work in this direction includes \cite{thurstonplus}, which (among other results) finds a spine for $\Poly_n(\C)$ consisting of squarefree polynomials all of whose critical values have modulus $1$; the method of proof passes through a consideration of the logarithmic derivative (see \cite[Theorem 9.2]{thurstonplus}).  McCammond states \cite[Remark 3.4]{mccsurvey} that similar ideas were known to Krammer. Dougherty--McCammond \cite{mcd} have investigated the combinatorial and topological structure of a polynomial map, obtaining something similar to the ``strip decomposition'' of \Cref{section:logdiff} (although without the perspective of the logarithmic derivative), and in forthcoming work \cite{mcd2} describe a cell structure on $\Poly_n(\C)$ that is compatible with the equicritical stratification. The ``strip decomposition'' models we consider here also bear some resemblance to B\"odigheimer's theory \cite{bodigheimer} of radial slit configurations as a configuration space model for the moduli space of Riemann surfaces with boundary.

\begin{remark}[A finer stratification?]
The equicritical stratification admits a further refinement where one tracks the multiplicities of both critical points and critical values, and it is natural to wonder about the corresponding questions on this finer stratification. As explained in \cite[Remark 1.4]{stratbraid1}, this in fact quickly reduces to classical considerations, since each of these finer strata is essentially just a Hurwitz space. Therefore, the basic topology (fundamental group, asphericality) of these strata is already understood, and so for this reason, we limit our interest here to the study of the equicritical stratification as we have defined it, in terms of the critical points alone.
\end{remark}

\begin{remark}[Finiteness properties via BNS invariants]
    It is perhaps initially surprising that \Cref{theorem:main} implies that the framed braid group $B_\kappa[\psi]$ is finitely generated, as infinite-index subgroups enjoy no {\em a priori} finiteness properties. We briefly record here an alternative argument that $B_\kappa[\psi]$ is finitely generated via the theory of the BNS invariant. 

    Define the {\em pure} framed braid group 
    \[
    PB_\kappa[\psi] := B_\kappa[\psi] \cap PB_{n+p}
    \]
    in the obvious way, and note that since this is of finite index in $B_\kappa[\psi]$, it is finitely generated if and only if $B_\kappa[\psi]$ is. As explained in \Cref{lemma:pbnormal}, $PB_\kappa[\psi]$ is normal and co-abelian in $PB_{n+p}$. Thus finite generation of $PB_\kappa[\psi]$ can be established by means of the BNS invariant of $PB_n$, which was determined by Koban-McCammond-Meier \cite{kmccm}. We do not wish to launch into a detailed digression on BNS invariants, but suffice it to say that it is simple to explicitly verify that $PB_\kappa[\psi]$ satisfies the BNS criterion \cite[Theorem 4]{strebel} for finite generation. Note, though, that such methods do not furnish an explicit generating set, as is implicit in the proof of \Cref{theorem:main}. 

    Finally, we mention that co-abelian subgroups of $PB_n$ with the further property of normality in $B_n$ were investigated in the recent work of Day--Nakamura \cite{daynak}.
\end{remark}

\para{Approach} To prove \Cref{theorem:main}, we make use of the basic machinery of geometric group theory, obtaining information about a group (particularly a set of generators) from an action on a connected graph. The graph we consider is defined in \Cref{section:connectivity} as the graph of {\em admissible root markings} (ARMs), written, for a partition $\kappa$ of $n-1$, as $\bM_\kappa$. Fix a marking of $n+p+1$ points on $S^2$, of which $n$ are called ``roots'', $p$ are called ``critical points'', and one is called ``$\infty$''. A {\em root marking} is a system of $n$ arcs on $S^2$ connecting $\infty$ to each of the roots, which can be realized disjointly except at the common endpoint $\infty$. Root markings are closely related to the {\em tethers} studied by Hatcher--Vogtmann \cite[Section 3]{HV}. The relative winding number function provides for a $\Z$-valued invariant of any such arc, and a root marking is said to be {\em admissible} if the winding number of each constituent arc is zero. In \Cref{section:connectivity}, we establish \Cref{lemma:armconnected}, which shows that $\bM_\kappa$ is connected. From here, we study the action of the framed braid group $B_\kappa[\psi_T]$ on $\bM_\kappa$ and use this to show that $B_\kappa[\psi_T]$ coincides with the subgroup of elements appearing in the image of $\rho: \sB_n[\kappa] \to B_\kappa[\psi_T]$ (that the image is {\em contained} in $B_\kappa[\psi_T]$ is not hard to show; see \Cref{prop:contained}). 

\para{Outline} In \Cref{section:framedbraid}, we define relative winding number functions and the associated framed braid groups, and establish a number of basic results about them. In \Cref{section:connectivity}, we turn to the graph of admissible root markings $\bM_\kappa$ and various derivatives, ultimately showing the connectivity result \Cref{lemma:armconnected} mentioned above. In \Cref{section:logdiff}, we recall the passage from a polynomial to a translation surface given by assigning $f$ to the translation surface for its logarithmic derivative $df/f$, and we describe the basic combinatorial features (``strip decomposition'') of such a surface. Finally in \Cref{section:mainproof}, we prove \Cref{theorem:main}, by studying the action of $\sB_n[\kappa]$ on $\bM_\kappa$, using explicit deformations of translation surfaces to realize a generating set for $B_\kappa[\psi_T]$.

\para{A note on the figures} The reader should be aware that the figures in the paper sometimes use color to convey information, although the author hopes they are capable of communicating the ideas of the paper even in black-and-white.

\para{Acknowledgements} The author thanks Kathryn Lindsey for insight into the paper \cite{thurstonplus}. Many thanks are due to anonymous referees whose comments greatly improved the exposition of the paper. Support from the National Science Foundation (grant DMS-2153879) is gratefully acknowledged.

\section{Framed braid groups}\label{section:framedbraid}

Here we introduce the main object of study in this paper, the {\em framed braid groups}. These are certain subgroups of the braid group on $S^2$ that preserve a structure known as a ``relative winding number function''. In \Cref{section:logdiff}, we will see that such structures naturally arise when considering the translation surface structures on the Riemann sphere arising from logarithmic derivatives of polynomials. 

\subsection{Basic working environment; relative winding number functions}
Here we recall the theory of relative winding number functions on the plane with marked points. We follow the treatment given in \cite[Section 4]{stratbraid1} with a slight upgrade in notation, systematically replacing subscripts of the form ``$n,p$'' from \cite{stratbraid1} with the more descriptive ``$\kappa$''.

        Let $n \ge 2$ be given, and let $\kappa = k_1\geq \dots \geq k_p$ be a partition of $n-1$. Let $\C_\kappa$ denote the surface $\CP^1$ with $n+p+1$ marked points. By abuse of terminology, $n$ of these are specified as ``roots'', $p$ are specified as ``critical points'', and the remaining point is specified as ``$\infty''$, even when these positions do not correspond to those of any polynomial. 

       We equip $\C_\kappa$ with a {\em weight function}
        \[
        w: \{z_1, \dots, z_{n+p+1}\} \to \Z
        \]
        from the set of distinguished points to $\Z$. The weight of each root and $\infty$ is $-1$, while the weights of the critical points are given (with respect to some ordering) as $k_1, \dots, k_p$. This is consistent with the order of vanishing/pole at the roots, critical points, and $\infty$ on the logarithmic derivative $df/f$.

        For the remainder of the paper, an integer $n \ge 2$ and a partition $\kappa = k_1 \geq \dots \geq k_p$ of $n-1$ with $p$ parts shall be fixed. The roots of $\C_\kappa$ will be enumerated as $z_1, \dots, z_n$, and the critical points will be enumerated as $w_1, \dots, w_p$.

           \begin{definition}[$\kappa$-marked braid group]
             The {\em $\kappa$-marked braid group} $B_\kappa$ is the subgroup of the spherical braid group $B_{n+p+1}(S^2):= \pi_1(\UConf_{n+p+1}(S^2))$ consisting of braids that fix $\infty$, preserve the set of roots setwise, and preserve setwise each set of critical points of a given order. Where convenient, we will write $PB_\kappa$ in place of the more cumbersome notation $PB_{n+p+1}(S^2)$ for the subgroup of $B_\kappa$ consisting of pure braids, i.e. the pure spherical braid group on $n+p+1$ strands.
        \end{definition}

        \begin{definition}[Relative winding number function, twist-linearity]
        Let $\cA_\kappa$ denote the set of isotopy classes of properly-embedded arcs in $\C_\kappa$ with one endpoint at $\infty$ and the other at a root (the tangent vectors at either end are not specified or required to be fixed under isotopy). A {\em relative winding number function}
        \[
        \psi: \cA_\kappa \to \Z
        \]  
        is a function subject to the following {\em twist-linearity property}: let $c\subset \C_\kappa$ be a simple closed curve, oriented so that $\infty$ lies to the right in the chosen direction of travel. Denote the disk bounded by $c$ to the left as $D$. Then for any $\alpha \in \cA_\kappa$,
        \[
        \psi(T_c(\alpha)) = \psi(\alpha) + \pair{c, \alpha}\left(1+ \sum_{p_i \in D} w(p_i)\right).
        \]
        Here, $T_c$ denotes the right-handed Dehn twist about $c$, the arc $\alpha$ is oriented so as to run from $\infty$ to a root, $\pair{\cdot, \cdot}$ denotes the relative algebraic intersection pairing, and the sum is taken over the subset of distinguished points lying in $D$. 
        \end{definition}

\begin{remark}
    The terminology suggests that $\psi(\alpha) \in \Z$ should be interpreted as some kind of winding number of $\alpha$ (``relative'' here indicates that winding numbers of {\em arcs}, as opposed to simple closed curves, are considered). Indeed, this will turn out to be the case for the ``logarithmic relative winding number function'' $\psi_T$ studied in \Cref{section:logdiff}, which will measure winding numbers of arcs relative to a certain background vector field studied therein. The twist-linearity condition axiomatizes how winding numbers change under the application of a Dehn twist. Winding number functions were introduced by Humphries--Johnson \cite{HJ}, who identified the essential role of the twist-linearity condition in axiomatizing functions that measure winding numbers of curves against a choice of vector field.
\end{remark}

                \begin{remark} {\em A priori}, there is a concern that $\psi$ could be ill-defined, on account of the fact that we have not specified tangent vectors for $\alpha$ at the endpoints. Thus $\alpha$ is isotopic to $T_c(\alpha)$, where $c$ is a simple closed curve enclosing one of the endpoints of $\alpha$. But since the order of each endpoint is $-1$, this is consistent with the requirements of the twist-linearity condition. For this same reason, we {\em cannot} measure winding numbers of arcs with an endpoint on a critical point, since there, winding numbers are not well-defined without working relative to a specified tangent vector.
        \end{remark}

        The following provides a useful criterion for computing winding numbers.

        \begin{lemma}[Computing $\psi$ via ``sliding'']\label{lemma:slide}
           Let $\alpha$ be an arc on $\C_\kappa$ connecting $\infty$ to a root $z$; let $\psi$ be a relative winding number function. Let $\beta$ be the arc obtained from $\alpha$ by sliding the right side across a distinguished point of weight $k$. Then $\psi(\beta) = \psi(\alpha) + k$.
        \end{lemma}
        \begin{proof}
            In this setting, $\beta = T_c(\alpha)$ for $c$ a simple closed curve enclosing $z$ and the distinguished point $w_i$ of weight $k$. The result now follows by twist-linearity. See \Cref{figure:sliding}.
        \end{proof}

\begin{figure}[ht]
\labellist
\tiny
\pinlabel \textcolor{myyellow}{$c$} at -5 10
\pinlabel \textcolor{myred}{$\alpha$} at 70 -5
\pinlabel \textcolor{myblue}{$\beta$} at 70 30
\pinlabel $w_i$ at 32 10
\endlabellist
\includegraphics[scale=1]{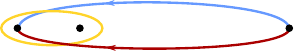}
\caption{Sliding the right side of $\alpha$ across $w_i$ adds $w(w_i) = k$ to $\psi(\alpha)$.}
\label{figure:sliding}
\end{figure}
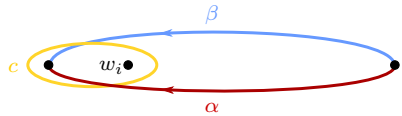
        
\subsection{Classification of relative winding number functions}

Our objective in this subsection is \Cref{lemma:specifyphi}, which shows that relative winding number functions on $\C_\kappa$ are in non-canonical bijection with $\Z^n$. This will require the following simple lemma.

        \begin{lemma}
            \label{lemma:phiequivariant}
            Let $\psi$ be a relative winding number function, and $\alpha, \beta \in \cA_\kappa$ be arcs with the same endpoints. If $\psi(\alpha) = \psi(\beta)$, then for $f \in PB_\kappa$ arbitrary, $\psi(f(\alpha)) = \psi(f(\beta))$.
        \end{lemma}
        \begin{proof}
            Since $\beta, \gamma \in \cA_\kappa$ have the same endpoints, they determine the same relative homology class and hence $\pair{c, \beta} = \pair{c,\gamma}$ for all simple closed curves $c$. It follows that for a pair of such arcs, $\psi(T_c(\beta)) = \psi(T_c(\gamma))$ for any simple closed curve $c$. Since $PB_\kappa$ is generated by Dehn twists \cite[Section 9.3]{FM}, inductively $\psi(f(\beta)) = \psi(f(\gamma))$ for any $f \in PB_\kappa$. 
        \end{proof}

        Our study of the framed braid group will revolve around systems of arcs with specified winding numbers, called {\em root markings}. Our first use for them will be to see that they suffice to characterize a given relative winding number function.

        \begin{definition}[(Partial) root marking, extension]
            A {\em root marking} of $\C_\kappa$ is a collection of $n$ arcs $A = \{\alpha_1, \dots, \alpha_n\}$ such that $\alpha_i$ begins at $\infty$ and terminates at the root $z_i$. The set of such $\alpha_i$ are required to be disjoint except at the common point at $\infty$. 

            A proper subset of arcs in a root marking is called a {\em partial root marking}. If $A'$ is a partial root marking, a root marking $A$ {\em extends} $A'$ if it contains $A'$ as a subset.
        \end{definition}

        \begin{lemma}\label{lemma:specifyphi}
            Let $\psi: \cA_\kappa \to \Z$ be a relative winding number function, and let $\{\alpha_1, \dots, \alpha_n\}$ be a root marking of $\C_\kappa$. Then $\psi$ is uniquely specified by the vector
            \[
            (\psi(\alpha_1), \dots, \psi(\alpha_n)) \in \Z^n,
            \]
            and conversely, any $v \in \Z^n$ arises in this way via some relative winding number function.
        \end{lemma}
        \begin{proof}
            As is well-known, the (spherical) pure braid group $PB_\kappa = PB_{n+p+1}(S^2)$ acts transitively on the set of isotopy classes of arcs with fixed endpoints (this is an instance of the ``change-of-coordinates principle'' of \cite[Section 1.3]{FM}). Thus the value of $\psi$ on any arc connecting $\infty$ to some root $z_i$ is determined by the value $\psi(\alpha_i)$, by the twist-linearity condition in conjunction with the fact that the pure braid group is generated by Dehn twists. 

            Conversely, we claim that given any $(x_1, \dots, x_n) \in \Z^n$, this is realized as $(\psi(\alpha_1), \dots, \psi(\alpha_n))$ for some relative winding number function $\psi$. One provisionally extends $\psi$ from $\{\alpha_1, \dots, \alpha_n\}$ to $\cA_\kappa$ by declaring $\psi(\beta)$ to be the value computed from the appropriate $\psi(\alpha_i)$ via the twist-linearity formula, and one seeks to verify that this is well-defined: if $f, g \in PB_{\kappa}$ satisfy $f(\alpha_i) = g(\alpha_i)= \beta$, must the value $\psi(\beta)$ as computed from $f$ agree with that given by $g$? Abusing notation, we will write $``\psi(f(\alpha_i))''$ to denote the value obtained by factoring $f$ into Dehn twists and repeatedly applying the twist-linearity formula.

            A first question is whether $\psi(f(\alpha_i))$ is even independent of the factorization of $f$ into twists. To do so, we will examine a presentation for $PB_\kappa = PB_{n+p+1}(S^2)$. Attaching a singly-punctured disk to the boundary of an $n+p$-punctured disk realizes $PB_{n+p+1}(S^2)$ as a quotient of the planar pure braid group $PB_{n+p+1}$ by the central twist $T_z$, where $z \subset \C$ is a curve separating $\infty$ from the remaining distinguished points \cite[Section 3.6]{FM}. Every relation in $\PB_{n+p+1}$ is a product of commutators \cite[Section 9.3]{FM}. It therefore suffices to show that (a) $\psi([T_c,T_d](\alpha)) = \psi(\alpha)$ for arbitrary curves $c,d$ on $\C_\kappa$, and, (b) 
            $T_z$ preserves winding numbers. (b) is easy to establish - the sum of the orders of the $n+p$ distinguished points enclosed by $z$ is $-1$ (being composed of $n$ roots of order $-1$ and $p$ critical points of orders $k_i$ summing to $n-1$), so by twist-linearity, $T_z$ has no effect on winding numbers.
            
            It remains to consider (a). Writing
            \[
            [T_c, T_d] = T_c T_{T_d(c)}^{-1},
            \]
            we find, by the twist-linearity formula,
            \[
            \psi([T_c,T_d](\alpha)) = \psi(T_{T_d(c)}^{-1}(\alpha)) + \pair{c,T_{T_d(c)}^{-1}(\alpha)}k,
            \]
            where $k$ is an integer determined by the orders of the distinguished points inside $c$ as in \Cref{definition:lrwnf}. Applying twist-linearity to the first term,
            \[
            \psi([T_c,T_d](\alpha)) = \psi(\alpha) - \pair{T_d(c), \alpha} k' + \pair{c,T_{T_d(c)}^{-1}(\alpha)}k,
            \]
            where likewise $k'$ is determined by the orders of the distinguished points inside $T_d(c)$.
            We claim that $k = k'$ and that $\pair{T_d(c), \alpha} = \pair{c,T_{T_d(c)}^{-1}(\alpha)}$. Both of these are true for the same reason: the curves $c$ and $T_d(c)$ enclose the same set of distinguished points (note that the algebraic intersection number is $1$ or $0$ depending on whether $\alpha$ terminates inside $c$ (equivalently, inside $T_d(c)$) or not).

            Having established that the value $\psi(f(\alpha_i))$ can be computed from $\psi(\alpha_i)$ via any factorization of $f$ into Dehn twists, we next suppose that $f, g \in PB_{n+p+1}$ satisfy $f(\alpha_i) = g(\alpha_i)$. We apply \Cref{lemma:phiequivariant} to see that $\psi(f(\alpha_i)) = \psi(g(\alpha_i))$ if and only if $\psi(\alpha_i) = \psi(f^{-1}g(\alpha_i))$. By assumption, $f^{-1}g$ fixes $\alpha_i$, and so can be viewed as an element of $PB_{n+p}(S^2) \le PB_{n+p+1}(S^2)$. Therefore, $f^{-1}g$ can be factored into generators for this subgroup, which consist of Dehn twists disjoint from $\alpha_i$. By the twist-linearity formula, no such twist has an effect on the winding number of $\alpha_i$, as required.
        \end{proof}

    \subsection{The framed braid group}\label{subsection:framedbraid}
    Here we come to the main definition of the paper, the framed braid group.

        \begin{definition}[Framed braid group]\label{definition:framedbraid}
        Let $\psi: \cA_\kappa \to \Z$ be a relative winding number function. The {\em framed braid group} $B_\kappa[\psi]$ is the subgroup of $B_\kappa$ consisting of $f \in B_\kappa$ for which
        \[
        \psi(f(\alpha)) = \psi(\alpha)
        \]
        for all $\alpha \in \cA_\kappa$. The {\em pure framed braid group} $PB_\kappa[\psi]$ is the intersection 
        \[
        PB_\kappa[\psi] = B_\kappa[\psi] \cap PB_{\kappa}.
        \]
        \end{definition}

                The following lemma gives a simple finite criterion for determining membership in $B_\kappa[\psi]$.
        \begin{lemma}
            \label{lemma:framedcriterion}
            An element $f \in B_\kappa$ is contained in $B_\kappa[\psi]$ if and only if, for any root marking $\{\beta_1, \dots, \beta_n\}$, there are equalities $\psi(f(\beta_i)) = \psi(\beta_i)$ for $i = 1,\dots, n$.
        \end{lemma}
        \begin{proof}
           By \Cref{lemma:specifyphi}, a relative winding number function is determined by its values on any root marking. By hypothesis, the winding number functions $\psi$ and $f^{-1}\cdot\psi$ (where $f^{-1} \cdot \psi(\alpha) = \psi(f(\alpha))$) take the same values on $\{\beta_1, \dots, \beta_n\}$, and hence are equal.
        \end{proof}

        Framed braid groups are not normal in $B_\kappa$ (a given $B_\kappa[\psi]$ is conjugated by $f \in B_\kappa$ to the potentially distinct group $B_\kappa[f \cdot \psi]$), but they are not so far off.

        \begin{lemma}\label{lemma:pbnormal}
            For any relative winding number function $\psi$, the pure framed braid group is normal in $PB_{\kappa}$, arising as the kernel of the map
            \begin{align*}
            \Delta_\psi: PB_{\kappa} &\to \Z^n\\
            f &\mapsto (\psi(f(\alpha_1)) - \psi(\alpha_1), \dots, \psi(f(\alpha_n)) - \psi(\alpha_n)),
            \end{align*}
            where $\{\alpha_1, \dots, \alpha_n\}$ is an arbitrary root marking. 
        \end{lemma}

        \begin{proof}
            The only point in question is that $\Delta_\psi$ is a well-defined homomorphism. To see that $\Delta_\psi$ does not depend on the choice of root marking, suppose $\alpha_i'$ is some other arc with the same endpoints as $\alpha_i$. Arguing as in \Cref{lemma:phiequivariant}, then $\alpha_i$ and $\alpha_i'$ determine the same relative homology class, and so $\psi(f(\alpha_i')) - \psi(\alpha_i') = \psi(f(\alpha_i)) - \psi(\alpha_i)$ as required.

            That $\Delta_\psi$ is a homomorphism is similarly easy to verify: one finds that the $i^{th}$ component of $\Delta_\psi(fg)$ is given by
            \[
            \psi(fg(\alpha_i)) - \psi(\alpha_i) = \psi(f(g (\alpha_i))) - \psi(g( \alpha_i)) + \psi(g (\alpha_i)) - \psi(\alpha_i).
            \]
            By the argument of the first paragraph, since $g \in PB_\kappa$, the first two terms are equal to the $i^{th}$ component of $\Delta_\psi(f)$, and the latter two visibly form the $i^{th}$ component of $\Delta_\psi(g)$.
            \end{proof}

\section{Admissible root markings}\label{section:connectivity}
        This section constitutes the technical heart of the paper. The main objective is \Cref{lemma:armconnected}, which establishes the connectivity of a family of graphs acted on by (subgroups of) the framed braid group. Vertices of these graphs correspond to systems of arcs on $\C_\kappa$ with prescribed winding number ({\em admissible root markings}). In \Cref{section:mainproof}, we will use these results to identify the framed braid group with the image of the monodromy map from the corresponding equicritical stratum of polynomials.

        From here to the end of the paper, let $\psi$ denote a fixed relative winding number function on $\C_\kappa$.

    \subsection{Graphs of admissible arcs}
        
        \begin{definition}[Admissible arc]
            An arc $\alpha \in \cA_\kappa$ is said to be {\em admissible} (tacitly with respect to $\psi$) if $\psi(\alpha) = 0$.
        \end{definition}
            
        \begin{definition}[Admissible root marking (ARM)]
        A root marking $A = \{\alpha_1, \dots, \alpha_n\}$ is {\em admissible} if each $\alpha_i$ is admissible. An admissible root marking will be abbreviated to ``ARM''. Likewise, a partial root marking $A'$ is admissible if $\psi(\alpha_i) = 0$ for all $\alpha_i \in A'$, abbreviated to a ``partial ARM''. 
        \end{definition}

        \begin{remark}
            \label{remark:armordersroots} An ARM $A = \{\alpha_1, \dots, \alpha_n\}$ endows the set of roots with a cyclic ordering, determined by the cyclic ordering of the tangent vectors at $\infty$ of the arcs $\alpha_i$ when realized disjointly. Except where otherwise specified, we will assume that each $\alpha_{i+1}$ is adjacent clockwise from $\alpha_i$. 
        \end{remark}

        \begin{definition}[Graph of ARMs]
            The {\em Graph of ARMs}, written $\bM_\kappa$, is the following graph:
            \begin{itemize}
                \item The vertices of $\bM_\kappa$ are the ARMs on $\C_\kappa$,
                \item ARMs $A = \{\alpha_1, \dots, \alpha_n\}$ and $A' = \{\alpha_1', \dots, \alpha_n'\}$ are connected by an edge in $\bM_\kappa$ if $\alpha_i = \alpha_i'$ for all but a single index $i_0$, and if $\alpha_{i_0}'$ is disjoint from $\alpha_{i_0}$ except at their common endpoints.
            \end{itemize}
        \end{definition}

        \begin{lemma}\label{lemma:armextends}
            Let $A' = \{\alpha_1, \dots, \alpha_k\}$ be a partial ARM. Then there is an extension of $A'$ to an ARM $A$.
        \end{lemma}
        \begin{proof}
            Certainly $A'$ extends to some root marking $A'' = \{\alpha_1, \dots, \alpha_k, \beta_{k+1}, \dots, \beta_n\}$; it remains to alter the arcs $\beta_j$ so as to set the winding numbers to zero. For $i = k+1, \dots, n$, let $c_i$ denote a simple closed curve in $\C_\kappa$, disjoint from all arcs in $A''$ except $\beta_i$, that encloses the root $z_i$ in addition to all of the critical points (and no other distinguished points). By the twist-linearity formula, $\psi(T_{c_i}(\beta_i)) = \psi(\beta_i) + n-1$, with the winding numbers of all other arcs in $A''$ left unchanged. Thus by repeated application of $T_{c_i}^{\pm 1}$, it can be arranged so that $2-n \le \psi(\beta_i) \le 0$ for $k+1 \le i \le n$. \Cref{figure:f1} then shows how by ``repositioning the basepoint'' of $\beta_i$, the winding number can be adjusted to zero.
        \end{proof}
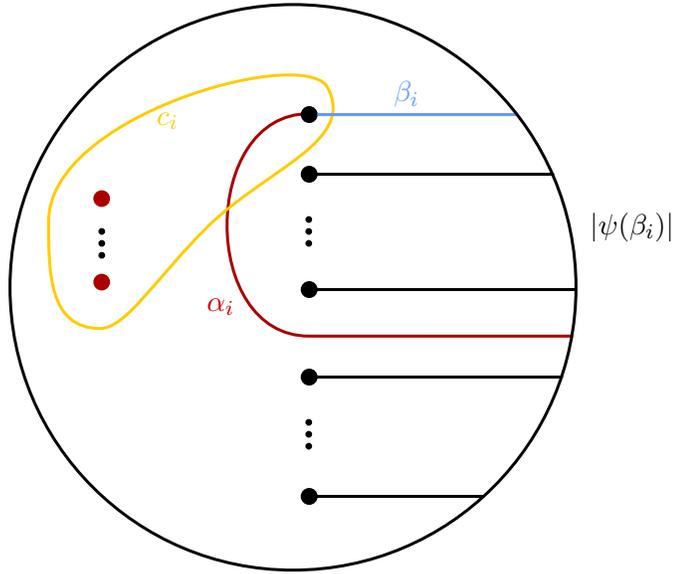
\begin{figure}[ht]
\labellist
\small
\pinlabel \textcolor{myblue}{$\beta_i$} at 150 180
\pinlabel \textcolor{myred}{$\alpha_i$} at 80 100
\pinlabel \textcolor{myyellow}{$c_i$} at 60 170
\pinlabel $\abs{\psi(\beta_i)}$ at 235 130
\endlabellist
\includegraphics[scale = 1]{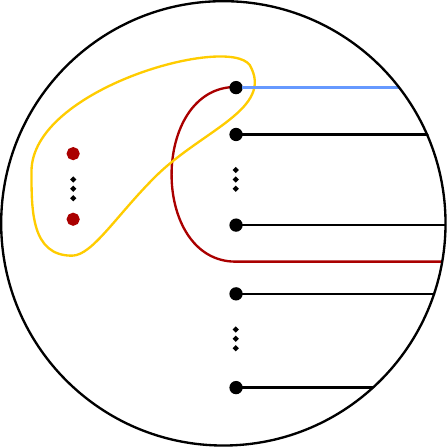}
\caption{The construction of \Cref{lemma:armextends}. Here we introduce some graphical conventions we will use throughout: $\C_\kappa$ will be denoted as a disk with the boundary collapsed to the point $\infty$, roots will be marked in black, and critical points will be marked in red. The ARM $A''$ is depicted as a collection of arcs from $\infty$ to the roots. There are $\abs{\psi(\beta_i)}$ zeroes in the region bounded by $\alpha_i, \beta_i$.}
\label{figure:f1}
\end{figure}

Our ultimate interest in in the connectivity of $\bM_\kappa$. To obtain this, it will be necessary to consider a family of auxiliary graphs.

     \begin{definition}[(Relative) Graph of admissible arcs]
            Let $A'$ be a partial ARM. We say that a root $z_i \in \C_\kappa$ is {\em marked} if some arc of $A'$ terminates at $z_i$, and is {\em unmarked} otherwise. The {\em graph of admissible arcs relative to $A'$}, written $\Adm_\kappa(A')$, is the following graph:
            \begin{itemize}
                \item The vertices of $\Adm_\kappa(A')$ consist of admissible arcs from $\infty$ to unmarked roots that are disjoint from $A'$ except at $\infty$,
                \item If there are at least two unmarked roots, then vertices $\alpha, \beta$ are connected by an edge in $\Adm_\kappa(A')$ if they are disjoint except at $\infty$ (and in particular, must terminate at distinct unmarked roots). If there is only one unmarked root, then $\alpha$ and $\beta$ are joined in $\Adm_\kappa(A')$ if they are disjoint except at both endpoints.
            \end{itemize}
        \end{definition}

\subsection{Connectivity of the graph of admissible arcs}

        \begin{lemma}
            \label{lemma:admconnected}
            Let $\kappa = k_1 \geq \dots \geq k_p$ be a partition of $n \ge 3$ with $p \ge 2$ parts. Let $A'$ be a partial ARM, possibly empty. Then $\Adm_\kappa(A')$ is connected.
        \end{lemma}
        This is the most intricate and technically demanding step of the argument. We will require three different arguments for three different regimes: the case of $A'$ empty, the case of $A'$ marking at most $n-2$ of the $n$ roots, and the case of $A'$ marking $n-1$ roots.

        \subsubsection{Case 1: $A'$ empty}
        The methods here will are reminiscent of other connectivity arguments used in the study of framed/$r$-spin mapping class groups, cf \cite[Section 7]{toric} and \cite[Section 5.3]{strata3}. The basic principle is to exploit the connectivity of a different graph of ``enveloping subsurfaces'' which is easier to establish, and then build a path in the original graph by exploiting existence results for objects of the desired type inside the enveloping surfaces.

        \begin{definition}[(Graph of) simple envelopes]
            A {\em simple envelope} on $\C_\kappa$ is a properly-embedded arc $E$ with both endpoints at $\infty$, such that on one side, $E$ encloses exactly two distinguished points, each of order $\pm 1$, at least one of which is a root (i.e. of order $-1$). For simplicity, we will think of $E$ as the boundary of this distinguished region.
            
            The {\em graph of simple envelopes} $\Env_\kappa$ is the graph with vertices given by isotopy classes of simple envelopes, with $E$ and $E'$ adjacent if they are disjoint except at $\infty$.
        \end{definition}

        \begin{lemma}
            \label{lemma:envconnected}
            Let $n \ge 3$ and let $\kappa = k_1 \geq \dots \geq k_p$ be a partition with $p \ge 2$ parts. Then $\Env_\kappa$ is connected.
        \end{lemma}
        \begin{proof}
            This will follow by an application of the {\em Putman trick} \cite[Lemma 2.1]{putmantrick}. This asserts the following: let $G$ be a group acting on a graph $X$ with generating set $S=S^{-1}$. Let $v$ be a vertex of $X$. Suppose that the $G$-orbit of every vertex intersects the connected component of $v$, and that for all $s \in S$, there is a path connecting $v$ to $s\cdot v$. Then $X$ is connected. 

            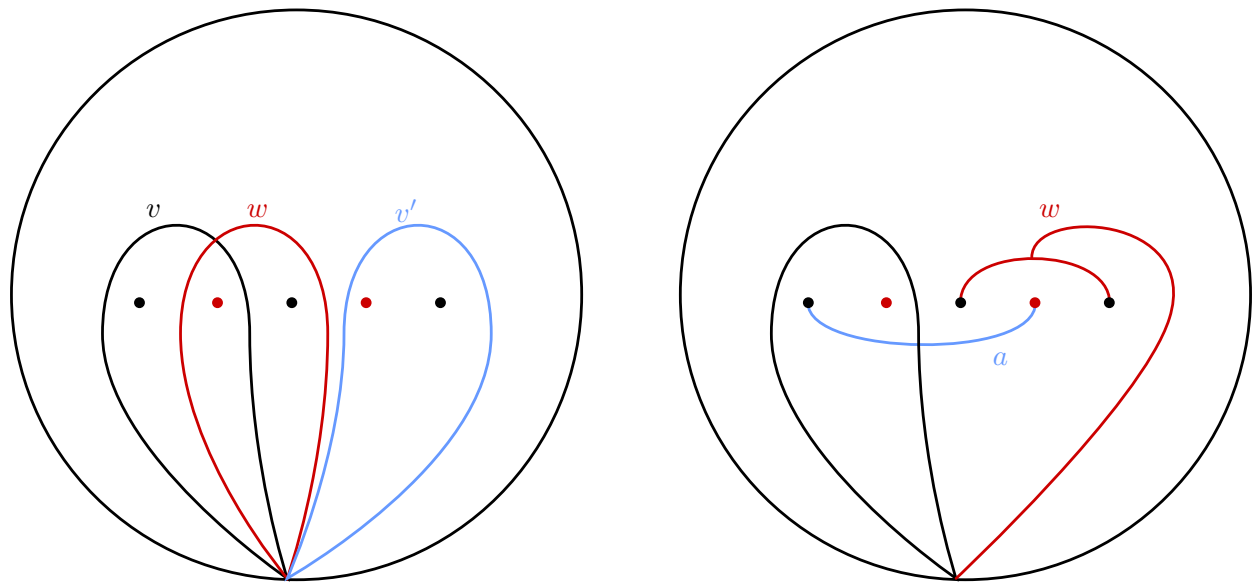
\begin{figure}[ht]
\labellist
\small
\pinlabel $v$ at 58 150
\pinlabel \textcolor{myred}{$w$} at 100 150
\pinlabel \textcolor{myblue}{$v'$} at 160 150
\pinlabel \textcolor{myred}{$w$} at 420 150
\pinlabel \textcolor{myblue}{$a$} at 400 90
\endlabellist
\includegraphics[width=\textwidth]{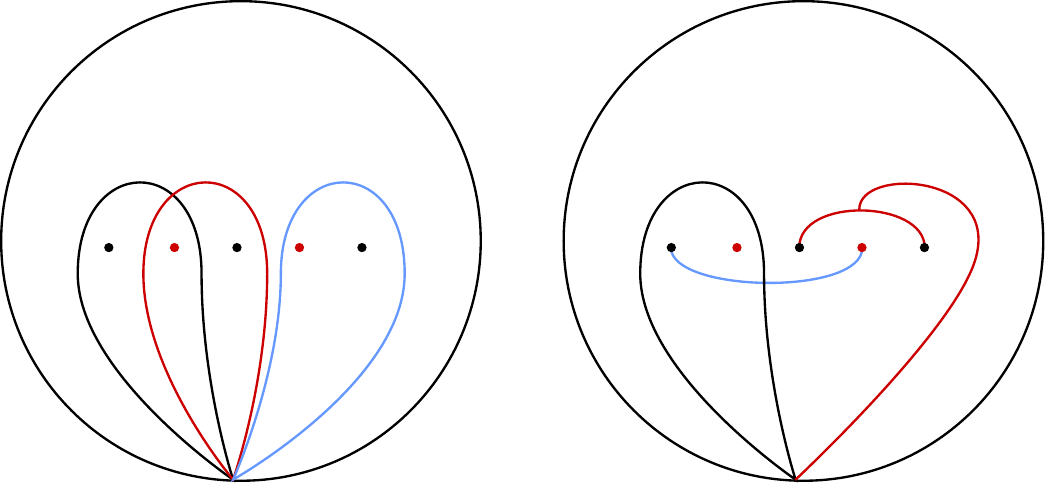}
\caption{We illustrate the arguments here in the maximally constrained case $n = 3, \kappa = \{1,1\}$. At left, showing that $v$ is connected to every orbit of $PB_\kappa$ on $\Env_\kappa$. The sequence $v, v', w$ is a path in $\Env_\kappa$; the arguments for other orbit types are analogous. At right, exhibiting a path connecting $v$ to $T_a(v)$, where $a$ is a neighborhood of the indicated arc. These are both disjoint from a regular neighborhood of $w$, which forms a simple envelope.}
\label{figure:putmantrick}
\end{figure}

            We consider $X = \Env_\kappa$ the graph of simple envelopes, and consider the action of $G = PB_\kappa$ on $\Env_\kappa$. Let us specify the basepoint vertex $v$; this will require special consideration in low-complexity cases. For $n \ge 5$, we take $v$ to be an envelope enclosing two roots. The remaining cases are $n = 3, \kappa=\{1,1\}$, and $n = 4$, $\kappa = \{2,1\}$ or $\{1,1,1\}$. In all of these cases, there is a critical point of order $1$, and we take $v$ to be an envelope enclosing a root and such a critical point.
            
            We first verify that $v$ is connected to a representative of every orbit of $\PB_\kappa$. By the change-of-coordinates principle, such an orbit is classified by the two enclosed distinguished points. If these points are all distinct, it is trivial to exhibit a simple envelope in the given orbit disjoint from $v$. Otherwise, there is exactly one point in common. By our choice of $v$, there is at least one root not enclosed by either envelope. It is again trivial to exhibit a simple envelope (containing at least one root) disjoint from $v$ and from the orbit representative $w$. See \Cref{figure:putmantrick}.

            The second condition to check is that $v$ can be connected to $s\cdot v$ for all $s \in S$. We take for $S$ the standard generating set for $\PB_\kappa$, consisting of Dehn twists in a neighborhood of a system of $\binom{n+p}{2}$ arcs, one for each pair of distinguished points. If the points enclosed by the twisting curve are disjoint from or coincide with those enclosed by $v$, then $s\cdot v = v$. Otherwise, they overlap in one point. As in the previous paragraph, there is at least one additional root, and then it is easy to exhibit a simple envelope $w$ disjoint from $v$ and from the support $a$ of the twist; see \Cref{figure:putmantrick}. Thus, $v, w, T_a(v)$ is a path in $\Env_\kappa$.
        \end{proof}

        \begin{lemma}
            \label{lemma:adminside}
            Let $E$ be a simple envelope. Then there is some admissible arc $\alpha$ contained on the distinguished side of $E$.
        \end{lemma} 
        \begin{proof}
            Let $\beta$ be any arc connecting $\infty$ to one of the roots contained on the distinguished side of $E$. Let $c$ be a curve contained inside $E$ enclosing both distinguished points, and with $\pair{c, \beta} = 1$. Then by twist-linearity, applying $T_c$ alters the winding number of $\beta$ by $\pm 1$, the sign being determined by the order of the other point. Thus, some twist $T_c^k(\beta)$ is admissible and contained inside $E$.
        \end{proof}

        \begin{proof}[Proof of \Cref{lemma:admconnected} for $A'$ empty]
            Let $\alpha, \beta \in \Adm_\kappa$ be given. Choose envelopes $E_\alpha, E_\beta \in \Env_\kappa$ containing $\alpha,\beta$, respectively. By \Cref{lemma:envconnected}, there is a path $E_\alpha = E_1, \dots, E_n = E_\beta$ in $\Env_\kappa$. By \Cref{lemma:adminside}, each $E_i$ for $1<i<n$ contains an admissible arc $\alpha_i$, and by construction, each $\alpha_i, \alpha_{i+1}$ are disjoint except at the common basepoint $\infty$. Thus $\alpha$ and $\beta$ are connected in $\Adm_\kappa$ via the path $\alpha = \alpha_1, \alpha_2, \dots, \alpha_n = \beta$.
        \end{proof}

    \subsubsection{Case 2: $A'$ nonempty, $\ge 2$ unmarked roots}        In the sequel, we will consider the intersection number of arcs that share one or more endpoint. As always, we define the {\em geometric intersection number} $i(\alpha,\beta)$ to be the minimal number of crossings as $\alpha, \beta$ range through their isotopy classes, keeping in mind that the tangent vectors of $\alpha, \beta$ at endpoints are not required to be fixed under isotopy.

    Ultimately, we will prove this case by induction on $i(\alpha,\beta)$. In preparation for this, we establish connectivity for small values of $i(\alpha, \beta)$.

        \begin{lemma}\label{lemma:i0sameendpt}
            Let $\kappa = k_1 \ge \dots \ge k_p$ be a partition of $n \ge 3$ with $p \ge 2$ parts. Let $A'$ be a partial ARM for which at least two roots are left unmarked. Let $\alpha, \beta \in \Adm_\kappa(A')$ be given with the same set of endpoints, with $i(\alpha, \beta) = 0$. Then $\alpha$ and $\beta$ are connected in $\Adm_\kappa(A')$. 
        \end{lemma}
        \begin{proof}
        We begin with a general observation. Let $\alpha, \beta$ be arcs with the same set of endpoints and with $i(\alpha, \beta) = 0$, but not necessarily admissible. Then $\alpha$ is isotopic to $\beta$ via an isotopy that drags $\alpha$ across each root or critical point enclosed by $\alpha \cup \beta$; each such point is crossed exactly once and with the same orientation. Via \Cref{lemma:slide}, $\abs{\psi(\alpha) - \psi(\beta)}$ is given as the sum of the weights of the enclosed points. 

         If moreover $\alpha, \beta$ is admissible, this shows that on each side of $\C_\kappa\setminus\{\alpha, \beta\}$, the sum of the orders of the enclosed critical points is equal to the number of enclosed roots. By assumption, there is at least one unmarked root on one side. \Cref{figure:f2} then shows how to construct $\gamma \in \Adm_\kappa(A')$ adjacent to both $\alpha, \beta$, via essentially the same construction as in \Cref{lemma:armextends}.
       \begin{figure}[ht]
\labellist
\small
\pinlabel \textcolor{myblue}{$\beta$} at 160 180
\pinlabel \textcolor{myred}{$\alpha$} at 80 90
\pinlabel \textcolor{mypurple}{$\gamma'$} at 160 160
\pinlabel \textcolor{mygreen}{$c$} at 83 138
\pinlabel \textcolor{myyellow}{$\gamma$} at 160 103
\pinlabel $\abs{\psi(\gamma')}<q$ at 240 130
\endlabellist
\includegraphics[scale=1]{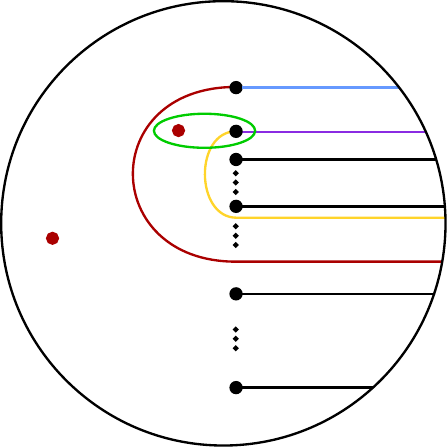}
\caption{The construction of \Cref{lemma:i0sameendpt}.}
\label{figure:f2}
\end{figure}
         
         Examining the figure, the total order of the critical points enclosed by $\alpha \cup \beta$ is $q$, for some $0 \le q \le n-1$. In fact, $0 < q < n-1$: were this not strict, $\alpha$ and $\beta$ would be isotopic, since, as remarked above, the total order of the critical points on either side equals the number of enclosed roots, so absence of one type of distinguished point enforces the absence of the other. There may be multiple critical points inside, but only one is illustrated here for clarity. Possibly some of the roots depicted as marked are in fact unmarked, but this has no effect on the argument. To construct $\gamma$, connect the free root to $\infty$ inside $\alpha \cup \beta$ by some arc $\gamma'$, then twist about $c$ as shown to arrange $1-q \le \psi(\gamma') \le 0$. By ``repositioning the basepoint'' as in \Cref{lemma:armextends}, an admissible arc $\gamma$ as shown can be constructed.
        \end{proof}
 
We will also need to examine connectivity for $i(\alpha, \beta) = 1$, subject to some special additional hypotheses (these will arise naturally in the inductive step).

        \begin{lemma}\label{lemma:i1sameendpt}
            Let $\kappa = k_1 \ge \dots \ge k_p$ be a partition of $n \ge 3$ with $p \ge 2$ parts. Let $A'$ be a partial ARM for which at least two roots are left unmarked. Let $\alpha, \beta \in \Adm_\kappa(A')$ be given with the same set of endpoints and $i(\alpha, \beta) = 1$, so that $\alpha \cup \beta$ divides $\C_\kappa$ into three components, two of which are bigons bounded by one segment each from $\alpha, \beta$. Suppose that on the interior of each of these bigons, there is exactly one root and no other distinguished point. Then $\alpha$ and $\beta$ are connected in $\Adm_\kappa(A')$. 
        \end{lemma}
                \begin{figure}[ht]
\labellist
\small
\pinlabel \textcolor{myblue}{$\beta$} at 160 160
\pinlabel \textcolor{myred}{$\alpha$} at 100 130
\pinlabel \textcolor{myred}{$k_p$} at 90 110
\pinlabel \textcolor{myyellow}{$\gamma$} at 57 120
\endlabellist
\includegraphics[scale=1]{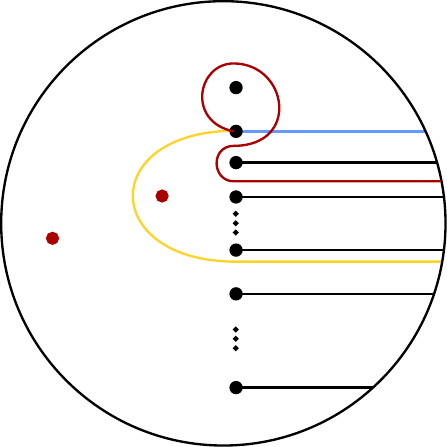}
\caption{The construction of \Cref{lemma:i1sameendpt}. }
\label{figure:f2point5}
\end{figure}
        \begin{proof}
          \Cref{figure:f2point5} shows how to construct $\gamma \in \Adm_\kappa(A')$ terminating at the same endpoint and with $i(\alpha,\gamma) = i(\beta, \gamma) = 0$. Examining the figure, we see that $\beta \cup \gamma$ encloses only the critical point of smallest order $k_p$ along with the root marked by $\alpha$ and $k_p-1$ other roots (marked or otherwise). Such $\gamma$ exists only if there are $k_p-1$ other roots available, distinct from the three roots involved in $\alpha\cup \beta$. This is always true: since $k_p$ is the smallest of at least two integers whose sum is $n-1$, necesarily $k_p \le n-2$. The result now follows from \Cref{lemma:i0sameendpt}.
        \end{proof}

        \begin{proof}[Proof of \Cref{lemma:admconnected}, $A'$ nonempty but $\ge 2$ unmarked roots]
        Let $\alpha, \beta \in \Adm_\kappa(A')$ be given. We will proceed by induction on $i(\alpha,\beta)$. First suppose that $\alpha, \beta$ terminate at the same root. If $i(\alpha, \beta) = 0$, then $\alpha,\beta$ are connected by \Cref{lemma:i0sameendpt}. Otherwise, replace $\beta$ with an admissible arc $\beta'$ terminating at a different root and adjacent to $\beta$ in $\Adm_\kappa(A')$ (such $\beta'$ always exists, by \Cref{lemma:armextends}). Thus we will assume in the sequel that $\alpha$ and $\beta$ terminate at different roots.

        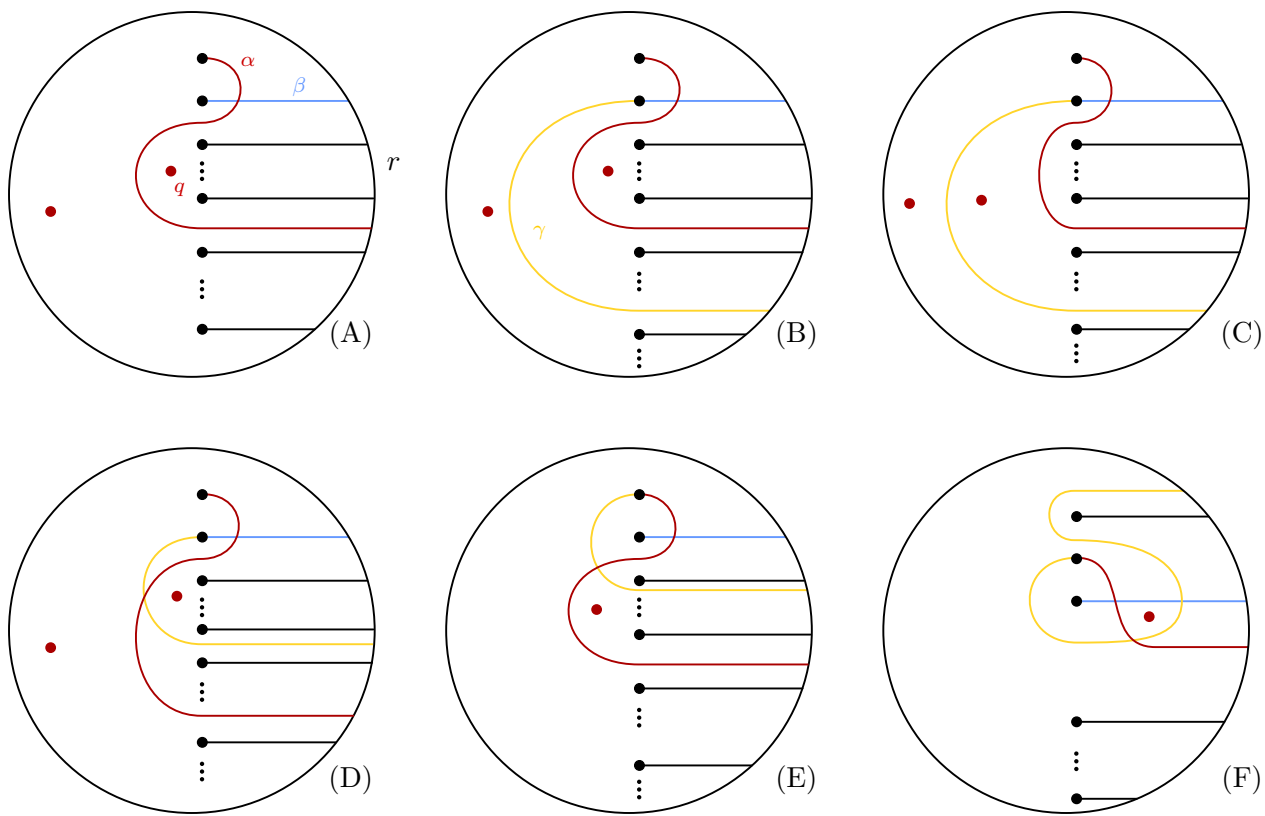
\begin{figure}[ht]
\labellist
\small
\pinlabel (A) at 200 280
\pinlabel (B) at 460 280
\pinlabel (C) at 720 280
\pinlabel (D) at 200 20
\pinlabel (E) at 460 20
\pinlabel (F) at 720 20
\pinlabel $r$ at 225 380
\tiny
\pinlabel \textcolor{myblue}{$\beta$} at 170 425
\pinlabel \textcolor{myred}{$\alpha$} at 140 440
\pinlabel \textcolor{myred}{$q$} at 100 365
\pinlabel \textcolor{myyellow}{$\gamma$} at 310 340
\endlabellist
\includegraphics[width=\textwidth]{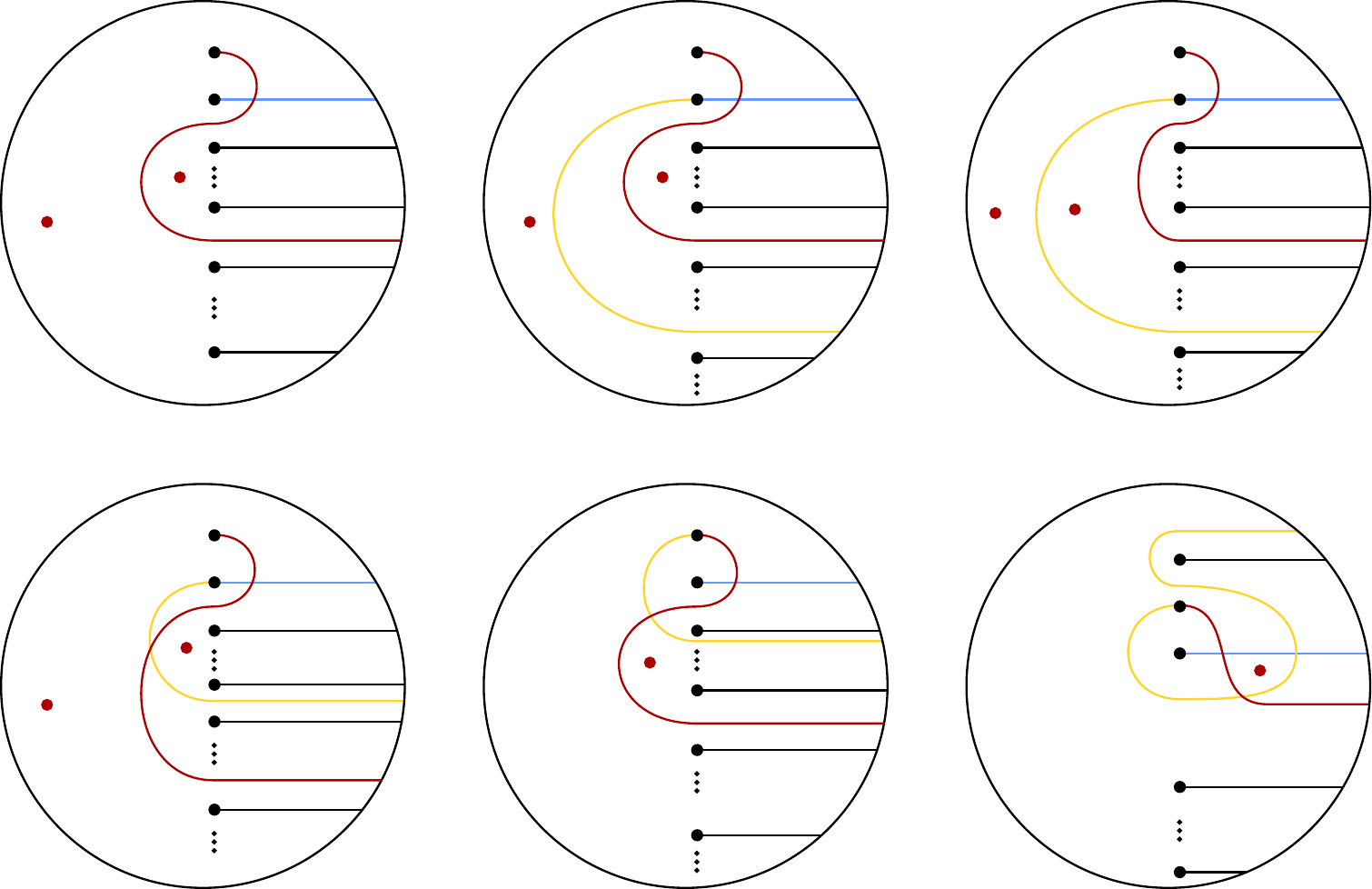}
\caption{The case $i(\alpha, \beta) = 1$.}
\label{figure:f3}
\end{figure}

        We take $i(\alpha, \beta) \le 1$ as base cases. In the case $i(\alpha,\beta) = 0$, the arcs are adjacent in $\Adm_\kappa(A')$. In the case $i(\alpha, \beta) = 1$, \Cref{figure:f3} shows how to connect $\alpha, \beta$. In (A), we see that, possibly after exchanging $\alpha, \beta$, the picture can be arranged so as to be of this form. The total order of critical points enclosed by $\alpha,\beta$ is $q$, and the number of roots enclosed is $r$. Panels (B)-(F) then consider various subcases depending on $q,r$: 
        \begin{enumerate}[(A)]
        \setcounter{enumi}{1}
        \item The case $r \le q \le n-2$. Here, admissible $\gamma$ can be constructed as shown, so that $\beta \cup \gamma$ encloses the same set of critical points and $q$ roots.
        \item The case $q = 0$. Here, $\gamma$ is constructed so as to enclose all but one critical point (of order $k$) along with $n-1-k \le n-2$ roots.
        \item (The case $0 < q < r$. Here, $\gamma$ is constructed so as to enclose the critical points of total order $q$ along with $q$ roots. $\alpha$ and $\gamma$ then satisfy the hypotheses of case (C), ultimately giving a path connecting $\alpha$ to $\beta$.
        \item The case $q = n-1$, $r >0$. Admissible $\gamma$ is constructed disjoint from $\beta$, and satisfying the hypotheses of \Cref{lemma:i1sameendpt} with $\alpha$.
        \item The case $q = n-1, r = 0$. Here, admissible $\gamma$ is constructed so that $\alpha, \gamma$ satisfy the hypotheses of \Cref{lemma:i1sameendpt}, and $\beta, \gamma$ belong to case (C).
        \end{enumerate}
        
        \begin{figure}[ht]
\labellist
\small
\pinlabel (A) at 200 20
\pinlabel (B) at 460 20
\pinlabel \textcolor{myred}{$\alpha$} at 170 50
\pinlabel \textcolor{myyellow}{$\gamma$} at 170 90
\pinlabel \textcolor{myblue}{$\beta$} at 175 140
\endlabellist
\includegraphics[width = \textwidth]{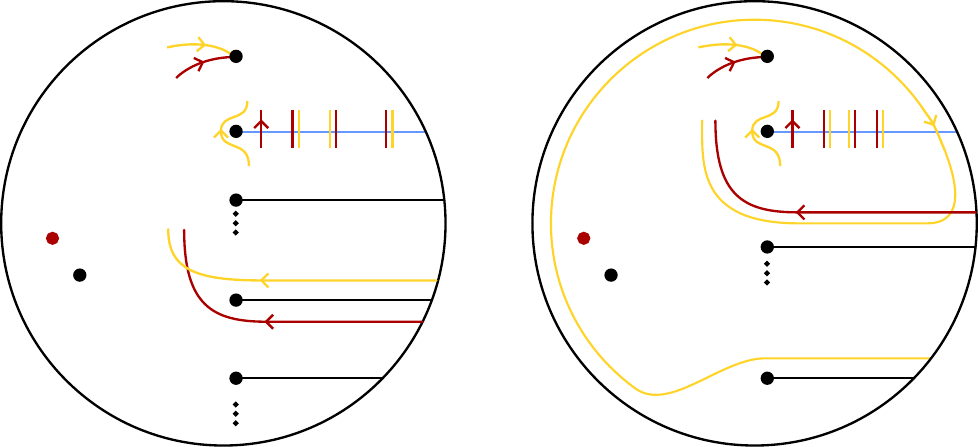}
\caption{Reducing intersection number.}
\label{figure:f4}
\end{figure}
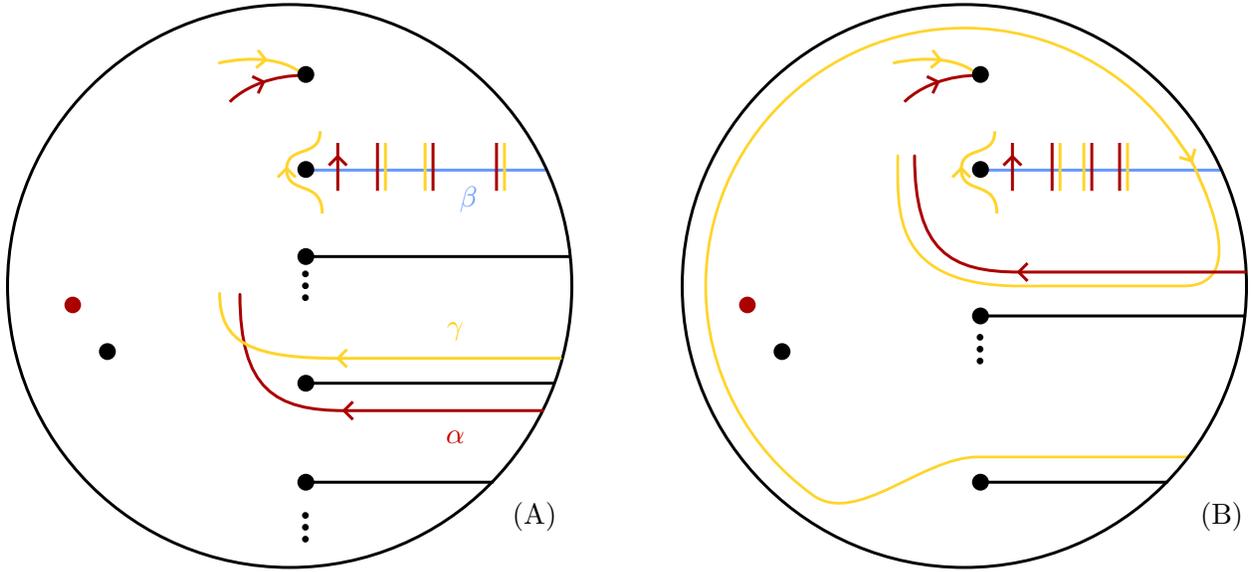

        We now assume that if $\xi, \eta \in \Adm_\kappa(A')$ terminate at distinct zeroes and satisfy $i(\xi, \eta) \le N$, then $\xi,\eta$ are connected in $\Adm_\kappa(A')$, and consider $\alpha, \beta \in \Adm_\kappa(A')$ with $i(\alpha, \beta) = N+1$. \Cref{figure:f4} shows how to construct $\gamma$ terminating at the same root as $\beta$, with $i(\beta, \gamma) = 1$ and satisfying the hypotheses of \Cref{lemma:i1sameendpt}, and with $i(\alpha,\gamma)< i(\alpha,\beta)$. The figure treats one possibility for the sign at the left-most intersection of $\alpha, \beta$; the other case is analogous. In (A): as long as there is a marked root lying in between $\alpha, \beta$ at $\infty$, admissible $\gamma$ can be constructed as shown, by sliding the left side of $\alpha$ across the endpoint of $\beta$ and compensating by sliding the right side (repositioning the base point) near $\infty$. By construction, $i(\beta,\gamma) < i(\alpha, \beta)$, and $\alpha,\gamma$ satisfy the hypotheses of \Cref{lemma:i1sameendpt}, hence are connected in $\Adm_\kappa(A')$. In (B), we consider the case where there is no such marked root to the right of $\alpha$. Here, construct admissible $\gamma$ as shown, by dragging the basepoint around a neighborhood of $\infty$ and repositioning. As before, $\alpha, \gamma$ satisfy the hypotheses of \Cref{lemma:i1sameendpt}, but here, $i(\beta, \gamma) = i(\alpha, \beta)$. However, one of the crossings has changed sign. Moving to the first crossing of $\beta, \gamma$, one repeats the argument. At some point, one will encounter a crossing of the opposite sign; the argument of (A) will then apply, decreasing intersection number.

        By the inductive hypothesis, $\alpha$ and $\gamma$ are connected in $\Adm_\kappa(A')$, and by \Cref{lemma:i1sameendpt}, $\beta$ and $\gamma$ are likewise connected, thus completing the inductive step.
        \end{proof}

\subsubsection{Case 3: one unmarked root}

        \begin{proof}[Proof of \Cref{lemma:admconnected}, $A'$ nonempty and one unmarked root]
        Let $\alpha, \beta \in \Adm_\kappa(A')$ be given, necessarily both terminating at the single unmarked root. We say that $\alpha, \beta$ {\em have coherent intersection} if there exist representatives for which every intersection has the same sign. To show that $\alpha, \beta$ are connected, we will proceed by induction on $i(\alpha, \beta)$, taking the case of coherent intersection as base case (note that if $i(\alpha,\beta) = 1$, then the intersection is necessarily coherent). 

\begin{figure}[ht]
\labellist
\small
\pinlabel \textcolor{myred}{$\alpha$} at 240 200
\pinlabel \textcolor{myyellow}{$\gamma$} at 210 60
\pinlabel $r$ at 300 127
\pinlabel ${\ge kq+r}$ at 318 60
\tiny
\pinlabel \textcolor{myred}{$q$} at 200 177
\pinlabel \textcolor{myred}{$kq+r$} at 270 160
\endlabellist
\includegraphics[scale=0.95]{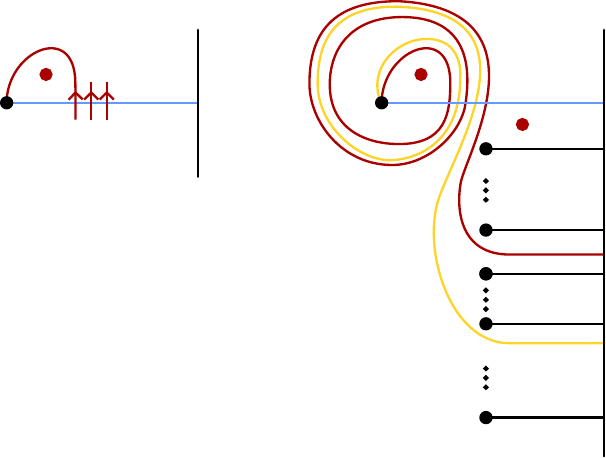}
\caption{Reducing intersection number, one unmarked root, coherent intersection.}
\label{figure:f5}
\end{figure}
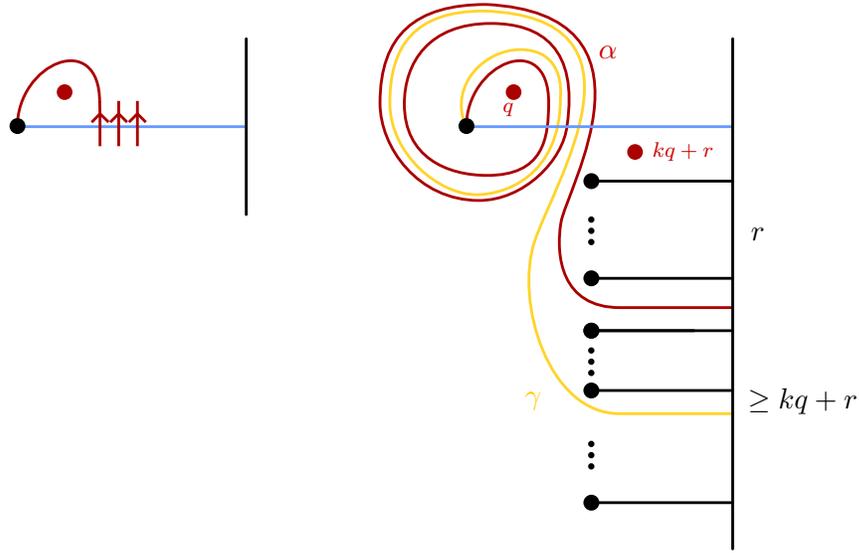

        \Cref{figure:f5} shows how to connect $\alpha, \beta$ in the case of coherent intersection. At left, we see that necessarily the leftmost crossing must be as shown (enclosing at least one critical point and terminating immediately at the unmarked root), as otherwise the leftmost crossing would have nowhere to go. Repeatedly applying this reasoning, we arrive at the global picture of $\alpha$ at right (possibly there are additional critical points inside the spiraling portion, but this does not affect the argument). Suppose the critical points enclosed by the bigon formed by the terminal segments of $\alpha, \beta$ have total order $q$, and that $i(\alpha, \beta) = k$. By twist-linearity, in order for $\alpha$ to be admissible, the region of $\C_\kappa$ enclosed by the initial segments of $\alpha, \beta$ contains $r \ge 0$ roots and one or more critical points of order totalling $kq+r$. Thus $n \ge (k+1)q + r +1$, so that there are at least $(k+1)q+r$ marked roots. It is therefore possible to construct admissible $\gamma$ as indicated by sliding the leftmost crossing over the bigon (reducing intersection number) and repositioning the basepoint to the left (down, in the figure) by $q$ positions. This completes the analysis in the case of coherent intersection.

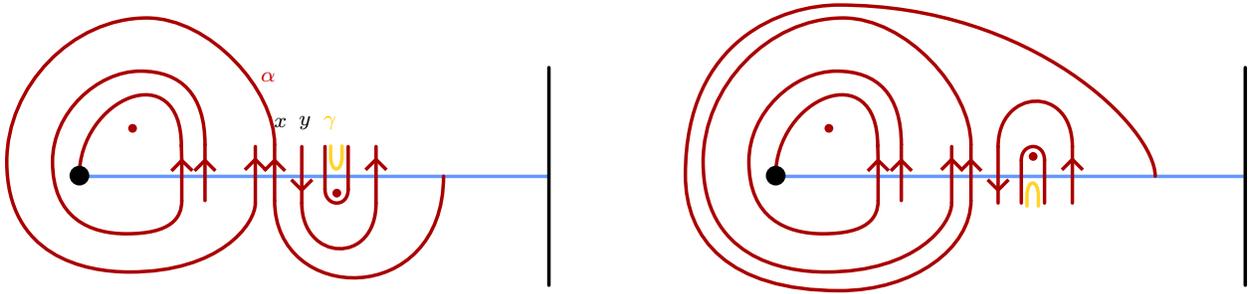
\begin{figure}[ht]
\labellist
\tiny 
\pinlabel \textcolor{myred}{$\alpha$} at 85 70
\pinlabel \textcolor{myyellow}{$\gamma$} at 105 55
\pinlabel $x$ at 89 55
\pinlabel $y$ at 97 55
\endlabellist
\includegraphics[width = \textwidth]{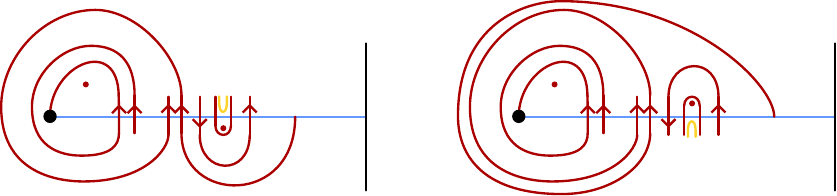}
\caption{Reducing intersection number, one unmarked root, incoherent intersection.}
\label{figure:f6}
\end{figure}
        For the general case, consider \Cref{figure:f6}. Let $y$ denote the first crossing from the left pointing opposite to the leftmost crossing, and let $x$ denote the crossing immediately to the left of $y$. \Cref{figure:f6} shows that regardless of the type of the segment feeding into $x$, one of the segments coming in or out of $y$ must bound a bigon. Passing to the innermost bigon, one constructs admissible $\gamma$ satisfying $i(\alpha, \gamma) = 0$ and $i(\beta, \gamma) < i(\alpha, \beta)$, by pulling $\alpha$ across all critical points in the innermost bigon and repositioning the basepoint of $\gamma$ as in \Cref{figure:f4} (not shown). As in that argument, it may be necessary to wrap $\gamma$ around the boundary of the disk, introducing another crossing with $\beta$. But since $i(\beta,\gamma)$ decreases by at least two by removing the bigon, so too in this case does $i(\beta, \gamma)$ strictly decrease. This completes the inductive step.
        \end{proof}
        
\subsection{Connectivity of the graph of ARMs}
        Having established the connectivity of graphs of a single arc at a time, we now consider the main graph that will feature in the proof of \Cref{theorem:main}, the graph of ARMs.

        \begin{definition}[Relative graph of ARMs]
            Let $A'$ be a partial ARM, possibly empty. The {\em graph of ARMs relative to $A'$}, written $\bM_\kappa(A')$, is the complete subgraph of $\bM_\kappa$ on vertices $A$ given as extensions of $A'$.
        \end{definition}

        \begin{proposition}\label{lemma:armconnected}
            Let $\kappa = k_1 \geq \dots \geq k_p$ be a partition of $n \ge 3$ with $p \ge 2$ parts. Let $A'$ be a partial ARM, possibly empty. Then the relative graph of ARMs $\bM_\kappa(A')$ is connected.
        \end{proposition}
        \begin{proof}
            We proceed by induction on the number $m$ of unmarked roots. In the base case $m = 1$, one verifies that the definitions of the graphs $\bM_\kappa(A')$ and $\Adm_\kappa(A')$ coincide, so that connectivity of $\bM_\kappa(A')=\Adm_\kappa(A')$ follows from \Cref{lemma:admconnected}.
            
            For the inductive step, for $k = n-m$, write $A' = \{\alpha_1, \dots, \alpha_k\}$, and let $A = \{\alpha_1, \dots, \alpha_n\}$ and $B = \{\alpha_1, \dots, \alpha_k, \beta_{k+1}, \dots, \beta_n\}$ be vertices of $\bM_\kappa(A')$. By \Cref{lemma:admconnected}, there is a path $\alpha_{k+1} = \gamma_0, \gamma_1, \dots, \gamma_q = \beta_{k+1}$ in $\Adm_\kappa(A')$, and as $m > 1$, each pair of successive admissible arcs $\gamma_i, \gamma_{i+1}$ terminate at distinct unmarked roots. For $0 \le i \le q-1$, define 
            \[
            A'_i = A' \cup \{\gamma_i, \gamma_{i+1}\}.
            \]
            By \Cref{lemma:armextends}, each $A_i'$ extends to an ARM $A_i$; also define $A_{-1} = A$ and $A_q = B$. 

            Each pair $A_i, A_{i+1}$ forms a pair of vertices in the relative graph $\bM_\kappa(A'\cup \{\gamma_{i+1}\})$. By induction, there is a path in $\bM_\kappa(A'\cup \{\gamma_{i+1}\})$ connecting $A_i$ to $A_{i+1}$; these paths can be concatenated to yield a path from $A = A_{-1}$ to $B = A_q$.
        \end{proof}

\section{Logarithmic derivatives and translation surface structures on the Riemann sphere}\label{section:logdiff}

In this section we recall the correspondence between polynomials and translation surface structures on the Riemann sphere. See also the treatment in \cite{stratbraid1}. 

\subsection{Polynomials and (root-labeled) translation surfaces}
Here we study the relationship between polynomials and the geometric world of translation surfaces. 

\para{From polynomials to translation surfaces..} Given $f \in \Poly_n(\C)[\kappa]$, we consider the logarithmic derivative $df/f$. This is a meromorphic differential on $\CP^1$ with $n+1$ simple poles at $\infty$ and at the distinct roots $z_1, \dots, z_n$. By the argument principle, the residues at the roots are each $2\pi i$, while the residue at $\infty$ is $-(2 \pi i)n$. The zeroes of $df/f$ occur at the critical points $w_1, \dots, w_p$ and have multiplicities $k_1 \ge \dots \ge k_p$ as specified by the partition $\kappa$.

Let $\mathcal{MD}(\kappa)$ denote the set of meromorphic differentials on $\CP^1$ with $n+1$ simple poles, $n$ of which have residue $2 \pi i$, and with $p$ zeroes of multiplicities specified by $\kappa$. According to \cite[Lemma 2.1]{stratbraid1}, every $\omega \in \mathcal{MD}(\kappa)$ is of the form $\omega = df/f$ for a uniquely-specified $f \in \Poly_n(\C)[\kappa]$, and there is an isomorphism of quasi-projective varieties
\[
\Poly_n(\C)[\kappa] \cong \mathcal{MD}(\kappa).
\]

Integration of $df/f$ endows $\CP^1$ (punctured at the roots of $f$ and $\infty$) with the structure of an infinite-area {\em translation surface} (for more on the basics of translation surfaces and their moduli spaces, see \cite[Section 3]{stratbraid1} or e.g. \cite{wright}). Let $\Omega_\kappa$ denote the moduli space of translation surfaces associated to $df/f$ for $f \in \Poly_n(\C)[\kappa]$ considered, as usual, up to cut-paste equivalence. As shown in \cite[Theorem 1.5]{stratbraid1}, $\Omega_\kappa$ is a complex orbifold of dimension $p-1$. We notate elements of $\Omega_\kappa$ as pairs $(T,\omega)$, where $T$ is a Riemann surface homeomorphic to an $n+1$-punctured sphere, and $\omega$ is a meromorphic differentials with the appropriate profile of poles, residues, and zeroes.

The orbifold structure can be understood explicitly as follows. The affine group $\Aff = \C \rtimes \C^*$ acts by precomposition on the space $\mathcal{MD}(\kappa)$ (or equivalently on $\Poly_n(\C)[\kappa]$), and for $n \ge 2$, all stabilizers are finite. Then $\Omega_\kappa$ is realized as the orbifold quotient
\[
\Omega_\kappa = \mathcal{MD}(\kappa)/\Aff \cong \Poly_n(\C)[\kappa]/\Aff.
\]

\para{...and back again} Translation surfaces are useful here because one can explicitly exhibit deformations (by moving the free prongs) and so probe the fundamental group. The drawback to working with $\Omega_\kappa$ directly is that it is an orbifold (and so one must work in the setting of the orbifold fundamental group), and in any event is a {\em quotient} of $\Poly_n(\C)[\kappa]$, so that there is potential ambiguity in lifting elements of the (orbifold) fundamental group.

To resolve this, we show here how to understand $\mathcal{MD}(\kappa)$ as a space of {\em root-labeled translation surfaces}. This will circumvent the need to reckon with orbifolds, while still allowing for the powerful deformation arguments available in the translation surface setting.

We temporarily pass to finite covers of $\mathcal{MD}(\kappa)$ and $\Omega_\kappa$. Define the cover
\[
\mathcal{MD}(\kappa)_2 = \left\{\left(\tfrac{df}{f}, z_1, z_2\right) \mid \tfrac{df}{f} \in \mathcal{MD}(\kappa),\ z_1,z_2 \in \C,\ z_1 \ne z_2, \ f(z_1) = f(z_2) = 0\right\}
\]
of differentials endowed with two distinguished roots. Likewise define $\Omega_{\kappa,2}$ as the cover of $\Omega_\kappa$ where two of the poles of residue $2 \pi i$ are distinguished. Note that this is a manifold (indeed, smooth variety) and not merely an orbifold, since the automorphism group of $\C$ marked at two points is trivial.

\begin{lemma}\label{lemma:coveriso}
    There is an isomorphism of complex manifolds
    \[
    RL: \mathcal{MD}(\kappa)_2 \to \Omega_{\kappa,2}\times \Conf_2(\C).
    \]
\end{lemma}
\begin{proof}
    There is a tautological assignment of $(df/f, z_1,z_2) \in \MD(\kappa)_2$ to the point
    \[
    ((\C \setminus Z(f), df/f, z_1, z_2),(z_1,z_2)) \in \Omega_{\kappa,2} \times \Conf_2(\C),
    \]
    where $Z(f) \subset \C$ denotes the roots of $f$.
    Conversely, suppose $((T,\omega, p_1, p_2),(z_1,z_2)) \in \Omega_{\kappa,2} \times \Conf_2(\C)$ is given. By basic complex analysis (see \cite[Lemma 2.1]{stratbraid1}), there is a polynomial $f$ and an isomorphism of translation surfaces
    \[
    \alpha: (\C \setminus Z(f), df/f) \to (T,\omega).
    \]
  Composing with the appropriate element of $\Aff$, there is a {\em unique} such $\alpha$ for which the distinguished points $p_1, p_2$ are identified with the chosen $z_1, z_2 \in \C$, defining the inverse map.
\end{proof}

As it stands, this identification requires the additional data of a choice of pair of points. This can be accounted for by introducing an equivalence relation on $\Omega_{\kappa,2} \times \Conf_2(\C)$. We define an equivalence relation as follows: 
\[
((T, \omega, p_1, p_2),(z_1,z_2)) \sim ((T', \omega', p_1', p_2'),(z_1',z_2'))
\]
if
\begin{enumerate}
    \item There is a (necessarily unique) conformal isomorphism of pointed translation surfaces 
    \[
    \iota: (T, \omega, p_1, p_2) \to (T',\omega',p_1',p_2'),
    \]
    and hence each of $(T, \omega)$ and $(T', \omega')$ are isomorphic to $(\C\setminus Z(f), df/f)$ for the same $df/f \in \MD(\kappa)$ (and $f(z_1) = f(z_2) = f(z_1') = f(z_2') = 0$),
    
    \item Under the unique map $\alpha: \C \setminus Z(f) \to T\cong T'$ identifying $z_1,z_2$ with $p_1, p_2$, also $z_1', z_2'$ are identified with $p_1',p_2'$.
\end{enumerate}
We call the equivalence classes {\em root-labeled translation surfaces}, and notate the space of such as
\[
\Omega_{\kappa}^{RL} := \left(\Omega_{\kappa,2}\times \Conf_2(\C)\right)/\sim.
\]

\begin{lemma}
    The equivalence class map
    \[
    \Omega_{\kappa,2}\times \Conf_2(\C) \to \Omega_\kappa^{RL}
    \]
    is a covering map, and the root-marking isomorphism $ RL: \mathcal{MD}(\kappa)_2 \to \Omega_{\kappa,2}\times \Conf_2(\C)$ of \Cref{lemma:coveriso} descends to an isomorphism
    \[
    \bar{RL}: \MD(\kappa) \to \Omega_\kappa^{RL}.
    \]
\end{lemma}
\begin{proof}
    Under the homeomorphism $RL^{-1}: (\Omega_{\kappa,2}\times \Conf_2(\C)) \to \MD(\kappa)_2$, a sufficiently small open set corresponds to a family of polynomials whose roots vary in pairwise-disjoint open sets of $\C$, with two of these neighborhoods distinguished by the given marking. In the topology on the covering space $\MD(\kappa)_2$, two such neighborhoods with different distinguished components are disjoint. The equivalence class of some $((T,\omega,p_1,p_2),(z_1,z_2))$ consists of the same translation surface $(T,\omega)$ with different pairs of poles marked by the roots of the underlying polynomial. Neighborhoods of each representative correspond under $RL^{-1}$ to the same set of polynomials but with different distinguished components, demonstrating the covering space condition. It is then straightforward to see that $RL$ preserves fibers of each of the covering maps, and so descends to the isomorphism
    \[
    \bar{RL}: \MD(\kappa) \to \Omega_\kappa^{RL}
    \]
    as claimed.
\end{proof}

\subsection{Strip decomposition} Translation surfaces in $\Omega_\kappa$ have a very simple structure. Following \cite[Section 3]{stratbraid1}, we recollect this here. The reader may wish to consult \Cref{figure:stripdecomp} while reading this discussion. Let $T \in \Omega_\kappa$ be a translation surface. $T$ carries a natural ``horizontal'' singular foliation induced by the kernel of the real $1$-form $\im(\tfrac{df}{f})$. All but finitely many leaves of the horizontal foliation for $T$ run from a zero of the associated polynomial $f$ to $\infty$; the finitely many exceptions have one or more endpoints at a cone point of $T$ (i.e. a critical point of $f$). A leaf with one or more endpoints at cone points is called a {\em prong leaf}. The closure of the set of leaves emanating from a chosen zero of $f$ is called a {\em strip}. As shown in \cite[Lemma 3.1]{stratbraid1}, $T$ is equal to the union of its strips, and each pair of strips intersect in finitely many prong leaves. 

Note that a given root $z$ of a polynomial $f \in \Poly_n(\C)[\kappa]$ has a canonically-associated strip on the translation surface for $\frac{df}{f}$: it is the closure of the set of leaves of the horizontal foliation that terminate at $z$.

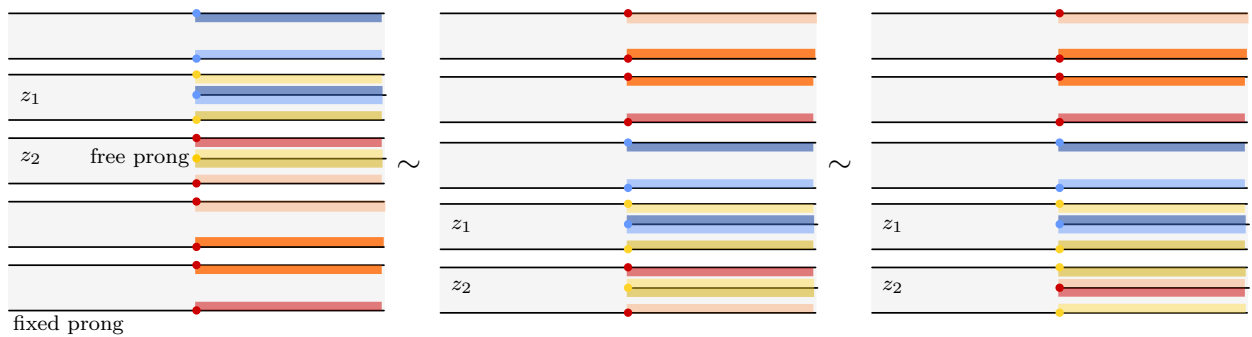
\begin{figure}[ht]
\labellist
\pinlabel $\sim$ at 265 100
\pinlabel $\sim$ at 550 100
\tiny
\pinlabel {fixed prong} at 40 -5
\pinlabel {free prong} at 87 105
\pinlabel $z_1$ at 15 145
\pinlabel $z_2$ at 15 105
\pinlabel $z_1$ at 300 60
\pinlabel $z_2$ at 300 20
\pinlabel $z_1$ at 585 60
\pinlabel $z_2$ at 585 20
\endlabellist
\includegraphics[width = \textwidth]{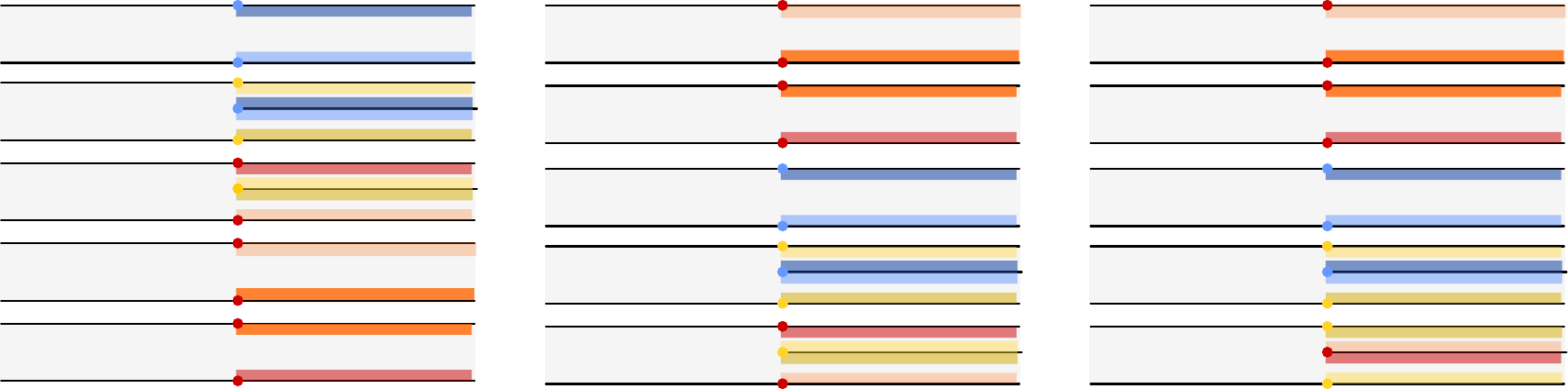}
\caption{Three depictions of the same root-marked translation surface in the stratum $\kappa = \{2,1,1\}$. The top left of each strip is glued to the bottom left, and gluing instructions on the right sides are indicated with colors. The solid colored dots in the middle of each strip are cone points for the flat metric, corresponding to critical points of the polynomial. The excess cone angle corresponds to the order of the critical point, so that e.g. the red point (with a total angle of $6 \pi = 2 \pi + 2(2 \pi)$) corresponds to the critical point of order $2$. From left to center, the strips have been vertically reordered. From center to right, the bottom strip has been recut, exchanging the roles of the fixed and free prongs. The root labeling data is indicated by labeling two distinguished points with $z_1, z_2 \in \C$.}
\label{figure:stripdecomp}
\end{figure}

 Every $T \in \Omega_\kappa$ admits a {\em strip decomposition} that depends on finitely many arbitrary choices. To define this, observe that each zero $z_i$ of $f$ must have at least one prong leaf emanating from $z_i$ (see \cite[Section 3.2]{stratbraid1}). Choosing one such leaf for each zero, $T$ is realizable as a union of $n$ bi-infinite strips in $\C$ of height $2 \pi$, where the top and bottom boundaries of each strip is given by the chosen prong leaf (technically, the boundary of a strip consists of {\em three} prong leaves - the one chosen leaf coming in from the root at left, along with two distinct prong leaves continuing on along the top and bottom to $\infty$ at right). The prongs comprising the boundary of a strip are called {\em fixed prongs}. 

There are then $p-1$ remaining prong leaves running from $\infty$ to a cone point, generically lying in the interior of the strips. These remaining prong leaves are called {\em free prongs}, as they are allowed to deform, changing the translation surface structure. As long as cone points do not collide, such a deformation stays in the stratum $\Omega_\kappa$. The relative periods of the $p-1$ free prongs is a set of local coordinates on $\Omega_\kappa$. When a strip contains multiple cone points, a cut-and-paste move can be used to exchange one free prong for a fixed prong, as illustrated in \Cref{figure:stripdecomp}.

\subsection{Monodromy of root-labeled translation surfaces} \label{subsection:RLmonodromy}
We established above the existence of isomorphisms
\[
\Poly_n(\C)[\kappa]\cong \MD(\kappa) \cong \Omega_\kappa^{RL}.
\]
Consequently there is a monodromy homomorphism
\[
\rho: \pi_1(\Omega_\kappa^{RL}) \to B_{n+p}
\]
pulled back from the natural monodromy map defined on $\pi_1(\Poly_n(\C)[\kappa])$. Here we explain how to compute the monodromy of a loop in $\Omega_\kappa^{RL}$ constructed as an explicit deformation. 

Let us recall the general principle of computing monodromy. Let $p: E \to S^1$ be a fiber bundle with fibers $S_t = p^{-1}(t)$ (for $t \in S^1 = \R/\Z$). Choose a {\em marking} $\mu_0: S \to S_0$ by some reference surface $S$. When, as in our setting, $S_0$ is a sphere marked at $n+p+1$ points (the roots and critical points of some $f$, and $\infty$), $\mu_0$ can be specified uniquely up to isotopy by a collection of $n+p$ disjoint arcs running from $\infty$ to each of the roots and critical points. Via parallel transport, the marking $\mu_0$ propagates to a family of markings $\mu_t: S \to S_t$, well-defined up to isotopy. The monodromy of the family is the map $\mu_1^{-1}\circ \mu_0: S \to S$; it is well-defined as an isotopy class.

There is a crucial subtlety in our setting that must be addressed. The braid group $B_{n+p}$ is defined as the fundamental group of the configuration space $\UConf_{n+p}(\C)$; this is the ultimate target of the monodromy map. However, the marking procedure above only recovers the {\em image} of the braid under the point-pushing homomorphism
\[
P: B_{n+p} \to \Mod(\C_\kappa).
\]
$P$ is not injective, and has infinite cyclic kernel $\pair{\Delta} = Z(B_{n+p})$, the center of $B_{n+p}$ generated by the full twist element $\Delta$. However, this is ultimately a non-issue, since the kernel $\pair{\Delta}$ is in the image of $\rho$, the generator being realized by applying the $S^1$-family of affine maps $z \mapsto e^{i \theta} z$ to any chosen basepoint.

There is an additional benefit to the {\em a priori} containment $\pair{\Delta} \le \im(\rho)$: in computing the mapping-class-group-valued monodromy, it is not actually necessary to track the root labeling data! As illustrated in \Cref{example1}, the root labeling data can become shuffled by cut/paste equivalence, so that in order to construct a closed loop in $\Omega_\kappa^{RL}$, it is necessary to combine the closed loop of {\em unmarked} translation surfaces with a path in the space of root markings. A choice of such path affects the resulting braid, but different choices differ by loops in the space of root markings. As we have seen, this image is contained in $\pair{\Delta}$. Thus, ignoring the root-marking data, the $B_{n+p}$-valued monodromy of a loop of {\em unmarked} translation surfaces in $\Omega_\kappa$ is well-defined up to the subgroup $\pair{\Delta} \le \im(\rho)$, which is all that is needed for our purposes. 
\\



\begin{example}\label{example1}
    Consider the family $T_t \subset \Omega_\kappa$ of translation surfaces shown in \Cref{figure:ex1}, for $\kappa = \{1,1\}$. The deformation proceeds by pushing the unique free prong up into the next vertically-adjacent strip, and then exchanging the bottom two strips. The marking propagates as shown. To understand this as a braid, identify $T$ with a five-punctured plane via the indicated marking. After moving through the loop, the new marking can be compared against the old, showing that the isotopy class induced by the braid is given as shown, by orbiting two of the roots about a fixed critical point at the center.
\end{example}

\begin{figure}[ht]
\labellist
\pinlabel $\sim$ at 310 260
\endlabellist
\includegraphics[width = \textwidth]{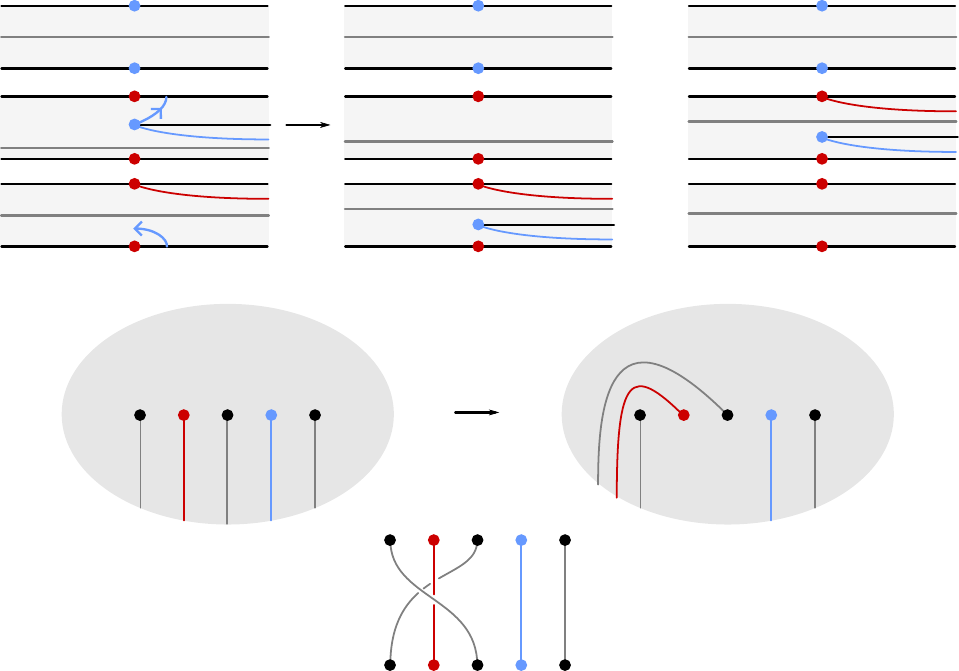}
\caption{Top row: the deformation of \Cref{example1}. The blue free prong is pushed up the top of the middle strip into the bottom, and then these two strips are switched, completing the loop. The blue arcs mark the roots, and the gray arcs mark the critical points. In the middle row, the corresponding marking of a punctured plane is illustrated at left; the effect on the marking is shown at right. The bottom row shows the corresponding braid.}
\label{figure:ex1}
\end{figure}

\subsection{Relative winding number functions on translation surfaces} The bridge connecting the work of \Cref{section:framedbraid,section:connectivity} to the present setting lies in the fact that translation surfaces in $\Omega_\kappa$ endow the marked Riemann sphere $\C_\kappa$ with a distinguished relative winding number function.

\begin{definition}[Logarithmic relative winding number function]\label{definition:lrwnf}
    Let $f_0 \in \Poly_n(\C)[\kappa]$ be chosen, and let $[((T_0,\omega_0,p_1,p_2),(z_1,z_2))] \in \Omega_\kappa^{RL}$ be the associated root-marked translation surface, chosen to lie outside the orbifold locus of $\Omega_\kappa$. As above, there is then a canonical identification $\alpha: \C_\kappa \to T_0$ given by integrating $\frac{df_0}{f_0}$. Under $\alpha$, properly-embedded arcs connecting $\infty$ to a root on $\C_\kappa$ are sent to bi-infinite arcs running from right to left on $T_0$, which can be isotoped so as to be eventually horizontal at both ends. The {\em logarithmic relative winding number function} $\psi_T: \cA_\kappa \to \Z$ is the relative winding number function on $\C_\kappa$ defined by measuring the winding number of the corresponding arc on $T_0$, relative to the horizontal vector field. It is straightforward to verify that $\psi_T$ satisfies the twist-linearity condition - see \cite[Lemma 4.2]{chillingworth2}.
\end{definition}

\section{Proof of \Cref{theorem:main}}\label{section:mainproof}
 Here we bring the settings of framed braid groups and equicritical strata together to prove \Cref{theorem:main}. Recall from the introduction that our ultimate interest is the monodromy representation
 \[
 \rho: \sB_n[\kappa] \to B_\kappa.
 \]
 Define the image
 \[
 \im(\rho) = \Gamma_\kappa \le B_\kappa.
 \]
 In \Cref{subsection:le}, we show that the translation surface structure for $df/f$ constrains $\Gamma_\kappa$ to lie in a framed braid group $B_\kappa[\psi_T]$. Then in \Cref{subsection:ge}, we show the opposite containment. Finally, in \Cref{subsection:maincor}, we improve the main theorem of \cite{stratbraid1}, giving a complete description of the monodromy of just the roots in an equicritical stratum.

\subsection{Monodromy lives in the framed braid group}\label{subsection:le}
\begin{proposition}\label{prop:contained}
    For all $n \ge 3$ and all partitions $\kappa = k_1 \ge \dots \ge k_p$, there is a containment
    \[
    \Gamma_\kappa \le B_\kappa[\psi_T],
    \]
    where $\psi_T$ is the logarithmic relative winding number function of \Cref{definition:lrwnf}.
\end{proposition}
\begin{proof}
This was established in \cite[Lemma 4.6]{stratbraid1}.
\end{proof}

\subsection{Monodromy equals the framed braid group}\label{subsection:ge}

Following the discussion of \Cref{subsection:RLmonodromy}, to compute the mapping class group-valued monodromy, it is not necessary to keep track of root-labeling information. Accordingly, we will suppress this throughout, working with unlabeled translation surfaces $\Omega_\kappa$.

\subsubsection{On basepoints} Our first task will be to describe a system of basepoints in $\Omega_\kappa$. It will be convenient to have a whole system of basepoints with different combinatorial properties, to account for all of the various possible ways that the $n$ strips can be attached via slits. Our basepoints will be indexed by {\em orderings} of the set $\{w_1, \dots, w_p\}$ of the critical points, which we enumerate as permutations $\sigma$ of the set $\{1, \dots, p\}$. We will also have occasion to consider a convenient admissible root marking $A_\sigma$ on $T_\sigma$.

\begin{construction}[Basepoint surface $T_\sigma \in \Omega_\kappa^{RL}$, admissible marking $A_\sigma$]\label{construction:T0}
Let $n \ge 2$ and $\kappa = k_1 \ge \dots \ge k_p$ be given, where $\kappa$ is a partition of $n-1$, indexing critical points $w_1, \dots, w_p$ of corresponding order. Let $\sigma$ be a permutation of $\{1, \dots, p\}$. We build $T_\sigma \in \Omega_\kappa$ starting with $n$ strips $S_1, \dots, S_n$, depicted in \Cref{figure:f8} as running from bottom to top. For $1 \le j \le k_{\sigma(1)}+1$, assign the fixed prong in $S_j$ to the critical point $w_{\sigma(1)}$ of order $k_{\sigma(1)}$. Subsequently, for $2 \le m \le p$, assign the fixed prong in strips $S_j$ for $k_{\sigma(1)} + \dots + k_{\sigma(m-1)}+2 \le j \le k_{\sigma(1)} + \dots + k_{\sigma(m)}+1$ to the critical point $w_{\sigma(m)}$ of order $k_m$. We call the set of strips with fixed prong assigned to $w_{\sigma(m)}$ the {\em $m^{th}$ group} of strips. Note the asymmetry in the construction: the first group of strips contains $k_{\sigma(1)}+1$ strips, while for $m \ge 2$, the $m^{th}$ group contains $k_{\sigma(m)}$ strips.

Next we specify the positions of the $p-1$ free prongs. By construction, there will be exactly one free prong assigned to each critical point $w_{\sigma(m)}$ for $m \ge 2$. Place a free prong at $\pi i$ in the topmost strip $S_{k_{\sigma(1)} + \dots + k_{\sigma(m-1)} + 1}$ in the $(m-1)^{st}$ group, assigned to $w_{\sigma(m)}$.

Finally, we specify the gluings. As always, glue the top left and bottom left half-edges of each $S_j$. In the first group, glue the top right of $S_j$ to the bottom right of $S_{j+1}$, including the top of $S_{k_{\sigma(1)}+1}$ to the bottom of $S_1$. For groups $2 \le m \le p$, likewise glue the top right of $S_j$ to the bottom right of $S_{j+1}$, but glue the top right of the topmost strip in the group to the bottom of the slit emanating from the free prong assigned to $w_{\sigma(m)}$, and the bottom right of the bottom-most strip in the group to the top of this slit. See \Cref{figure:f8}. This also depicts the admissible root marking $A_\sigma$ on $T_\sigma$, consisting of a system of admissible arcs at height $\epsilon$ in each of the strips.
\end{construction}

\begin{figure}[ht]
\labellist
\tiny
\endlabellist
\includegraphics[scale=0.95]{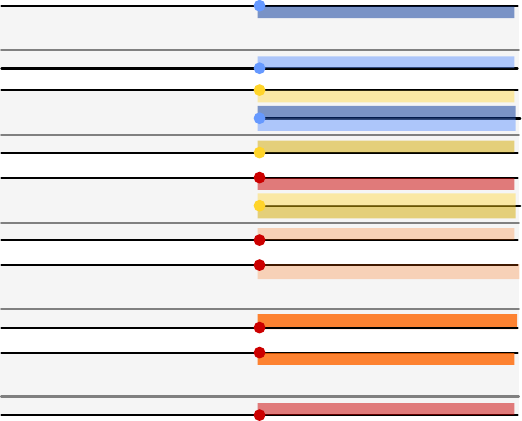}
\caption{The basepoint surface $T_\sigma \in \Omega_\kappa$ of \Cref{construction:T0}, illustrated in the case $n = 5, \kappa = \{2,1,1\}$, ordered so that $w_1$ has order $k_1 = 2$ and is colored red, $w_2$ has order $k_2 = 1$ and is yellow, and $w_3$ has order $k_3 = 1$ and is blue. The gluing instructions on the right halves of the surface are illustrated with matching colors. The admissible root marking $A_\sigma$ is shown as the system of horizontal arcs in gray.}
\label{figure:f8}
\end{figure}

\subsubsection{The half-push move and its consequences} Here we describe an extremely useful family of deformations in $\Omega_\kappa$, the {\em half-push}.

\begin{construction}[Half-push, full push]
Let $w$ and $w'$ be distinct cone points on $T \in \Omega_\kappa$. Suppose that in some strip $S_j$, the fixed prong is assigned to $w$, and there is some small neighborhood of $w$ whose intersection with $S_j$ contains a free prong for $w'$ near the top of the strip, and otherwise contains no other distinguished point. The {\em half-push} deformation takes this free prong for $w'$ and pushes it up the top-right side of $S_j$ into the next vertically-adjacent strip. This new strip is then re-cut so that the prong for $w'$ becomes fixed; the prong for $w'$ (and any other free prongs that might be present) becomes free. A {\em full push} of $w$ about $w'$ is the composition of two half-pushes, the first as described above, and the second with the roles of $w$ and $w'$ reversed. See \Cref{figure:halfpush,figure:halfpush2}.    
\end{construction}

\begin{figure}[ht]
\labellist
\endlabellist
\includegraphics[scale=0.95]{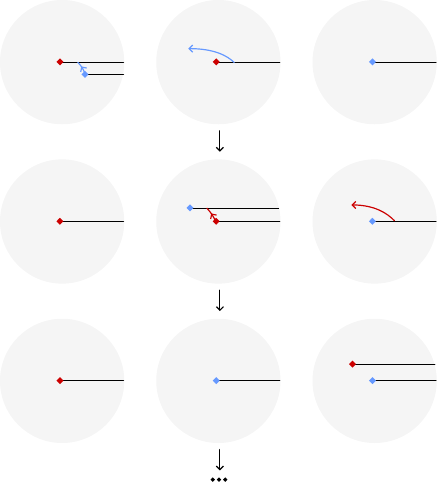}
\caption{A sequence of two half-pushes, local structure. Each row depicts three portions of neighborhoods of the red/blue cone points on a fixed translation surface.}
\label{figure:halfpush}
\end{figure}

\begin{figure}[ht]
\labellist
\endlabellist
\includegraphics[scale=0.95]{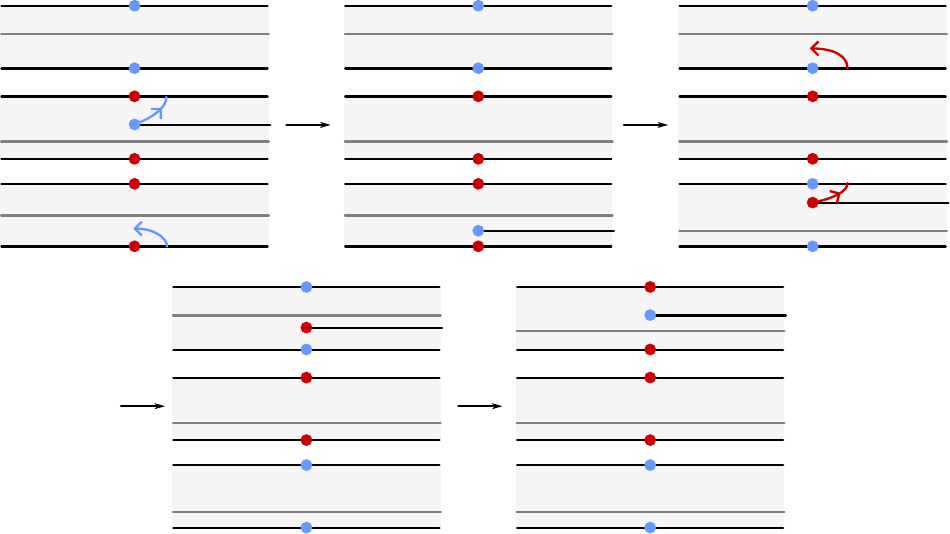}
\caption{A sequence of two half-pushes, with full strips depicted. We alternate between pushing and re-cutting.}
\label{figure:halfpush2}
\end{figure}

Using half-pushes, the following lemma will allow us to freely switch between basepoints as convenient.

\begin{lemma}
    \label{lemma:basepaths}
    For any two permutations $\sigma, \tau$ of $\{1,\dots, p\}$, there is a path in $\Omega_\kappa$ from $T_\sigma$ to $T_\tau$ that takes $A_\sigma$ to $A_\tau$.
\end{lemma}
\begin{proof}
    It suffices to construct such a path in the case when $\tau$ differs from $\sigma$ by a single transposition of adjacent elements. By hypothesis, there are critical points $w, w'$ that are adjacent on both $T_\tau$ and $T_\sigma$, and for which $w$ is below $w'$ on $T_\sigma$ and above $w'$ on $T_\tau$. By \Cref{construction:T0}, on $T_\sigma$ there is a free prong in the topmost slit for $w$ that is associated to $w'$. This can be pushed up on the associated slit until it is just below the fixed prong, leaving $A_\sigma$ undisturbed in the process. The local picture near the cone points for $w, w'$ now looks like that of \Cref{figure:halfpush}. Let $w$ have order $k$ and $w'$ order $k'$. After a sequence of $2k' + 1$ half-pushes, $w$ and $w'$ will have switched places: the strips formerly occupied by $w$ will now be occupied by $w'$ and vice versa. Vertically re-ordering the strips, one obtains a cut-paste equivalence realizing the resulting surface as $T_\tau$. At no point in this process did the deformations disturb the root marking, and so we see this path takes $A_\sigma$ to $A_\tau$ as required.
\end{proof}




We can also obtain some explicit monodromy elements via a sequence of half-pushes.

\begin{lemma}
    \label{lemma:fulltwist}
    Let $w, w'$ be critical points on a translation surface $T \in \Omega_\kappa$ such that some cone point for $w'$ lies in a small neighborhood of a cone point for $w$. Then there is a loop in $\Omega_\kappa$ based at $T$ inducing the full twist about the curve enclosing $w, w'$ contained in the union of small neighborhoods for $w, w'$.
\end{lemma}
\begin{proof}
     We saw in \Cref{lemma:basepaths} that a sequence of $2k' + 1$ half-pushes exchanges the roles of $w$ and $w'$. Consequently, performing a total of $2(k+k' + 1)$ half-pushes will exchange these once again. The monodromy of the resulting loop is a full twist of $w$ about $w'$, as topologically, we see the cone points $w$ and $w'$ orbit once around each other in a small neighborhood, leaving the rest of the surface undisturbed.
\end{proof}

\subsubsection{Containment of vertex stabilizer} The proof of \Cref{theorem:main} follows standard principles in geometric group theory (cf. \cite[Lemma 4.10]{FM}): given an action of a group $G$ on a connected graph $X$, one sees that $G$ is generated by elements taking a chosen vertex $v$ to adjacent vertices, along with the stabilizer $G_v$ of $v$. We will apply this principle to the action of $B_{\kappa}[\psi_T]$ on $\bM_\kappa$. For basepoint, we take $T_0 \in \Omega_\kappa$ to be the basepoint surface of \Cref{construction:T0} associated to the standard ordering $k_1, k_2, \dots, k_p$, and let $A_0$ be the associated ARM arising from \Cref{construction:T0}. We will take the vertex of $\bM_\kappa$ associated to $A_0$ as our basepoint. To see that $\Gamma_\kappa = B_\kappa[\psi_T]$, we will show that both types of generating elements are contained in the monodromy subgroup $\Gamma_\kappa$. In \Cref{lemma:monodromycontainsstab}, we consider the vertex stabilizer; in preparation, we first establish a fact about braid groups.

\begin{lemma}\label{lemma:coloredbraidgen}
    Let $\lambda: [n] \to \Z$ be a ``coloring'' of the finite set $[n] = \{1,2, \dots, n\}$, and let $B_n[\lambda] \le B_n$ be the ``colored braid group'' consisting of all braids that preserve the coloring $\lambda$. With respect to the standard generators $\{a_{i,j}\mid 1 \le i < j \le n\}$ of the pure braid group \cite[Section 9.3]{FM}, let $\sigma_{i,j}$ denote the corresponding half-twist. Then $B_n[\lambda]$ is generated by the set of elements
    \[
    \{\sigma_{i,j}^{w(i,j)} \mid 1 \le i < j \le n\},
    \]
    where $w(i,j) = 1$ if $\lambda(i) = \lambda(j)$ and $w(i,j) = 2$ otherwise.
\end{lemma}

\begin{proof}
    Let the subgroup generated by the indicated elements $\sigma_{i,j}^{w(i,j)}$ be denoted by $\Gamma$.
    Since $PB_n \le B_n[\lambda]$ and the generating set for $\Gamma$ evidently contains the generating set $\{\sigma_{i,j}^2 \mid 1 \le i < j \le n\}$ of $PB_n$, it remains only to show that the quotients $\Gamma/ PB_n$ and $B_n[\lambda]/PB_n$ are isomorphic. The latter decomposes as a product of symmetric groups on the points of a given weight, while the former includes all transpositions between points of equal weight; the result follows.
\end{proof}

        \begin{lemma} \label{lemma:monodromycontainsstab}
            Let $G_0 \le B_\kappa[\psi_T]$ denote the stabilizer of $A_0$. Then $G_0 \le \Gamma_\kappa$.
        \end{lemma}
        \begin{proof}
            A given $g \in G_0$ preserves $A_0$ as a set, but does not necessarily fix each individual arc. We first reduce to this case. According to \Cref{remark:armordersroots}, $A_0$ induces a cyclic ordering on the roots, which must be preserved by $g \in G_0$. Consider the deformation on $T_0$ induced by performing a full push of $w_2$ about $w_1$, then a full push of $w_3$ about $w_2$ and so on, up through a full push of $w_{p-1}$ about $w_p$. The effect is to simply vertically re-order the strips in $T_0$, pushing each one up by one (including taking the top strip $S_n$ down to the bottom). Such a deformation preserves $A_0$ while inducing a cyclic permutation on the roots. 
            
            Thus, by applying some number of these deformations, we can assume that $g$ preserves each arc of $A_0$. Let this subgroup be denoted $PG_0 \le G_0$; it remains to show that $PG_0 \in \Gamma_\kappa$. By definition, $PG_0$ consists of all braids that preserve each arc $A_0$ as well as the winding numbers of each component of $A_0$. But by \Cref{lemma:framedcriterion}, it follows that any braid preserving $A_0$ necessarily preserves $\psi_T$. Thus, $PG_0$ can be identified with the full ``colored braid group'' on the critical points on the cut-open surface $\C_\kappa \setminus A_0$ (we color the critical points according to their orders, and allow any braid that preserves color). Any such braid can be induced by a deformation of translation surfaces in $\Omega_\kappa$, as follows. 

\begin{figure}[ht]
\labellist
\endlabellist
\includegraphics[scale=0.95]{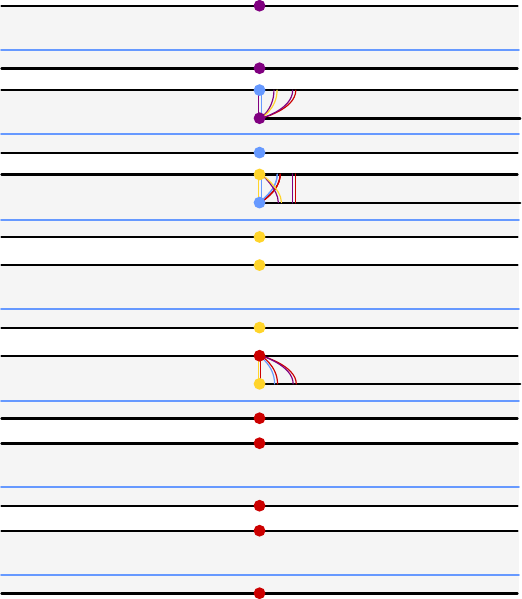}
\caption{A system of paths on $T_0$ avoiding $A_0$, illustrated in the case $n = 7, \kappa = \{2,2,1,1\}$ (it should be possible to infer the general construction from this). Gluing instructions have been suppressed for clarity but follow what was established in \Cref{construction:T0}. The path connecting $w$ to $w'$ is bicolored by the associated colors of $w$ and $w'$.}
\label{figure:pushsystem}
\end{figure}

            By \Cref{lemma:coloredbraidgen}, the colored braid group is generated by $\binom{p}{2}$ twists or half-twists about a suitable collection of arcs connecting each pair of critical points. \Cref{figure:pushsystem} illustrates a system of such arcs on the basepoint surface $T_0$. They form the intersection pattern of the standard generators for the pure braid group. Observe that each path is based at one end at a free prong. The Dehn twist about a neighborhood of the path is then realized in $\Omega_\kappa$ by pushing the free prong along the path to the other critical point, completing one full orbit (following \Cref{lemma:fulltwist}), and returning. 
            
            It remains to exhibit a half-twist exchanging critical points of the same order. It suffices to consider the case when critical points $w_m$ and $w_{m+1}$ are vertically adjacent in $T_0$, since the braid group is generated by half-twists about such adjacent points. If each of $w_m, w_{m+1}$ have order $k$, then the construction of \Cref{lemma:basepaths} in fact yields a {\em loop} based at $T_0$ which exchanges $w_m, w_{m+1}$ in a half-twist.
        \end{proof}

\subsubsection{Moving to adjacent vertices} Finally we show that the action of $\Gamma_\kappa$ on $\bM_\kappa$ is transitive on the vertices adjacent to the chosen basepoint. We first classify the orbits of adjacent vertices.

        \begin{definition}
            Let $A= \{\alpha_1, \dots, \alpha_n\}$ be an ARM on $\C_\kappa$, and let $\gamma$ be an admissible arc on $\C_\kappa$ disjoint from $A$ except at endpoints (necessarily $\infty$ and some root $z_i$). Then $\gamma \cup \alpha_i$ forms a simple closed curve, and every distinguished point except $z_i$ and $\infty$ lies in the interior of one of the two disks on $S^2$. This partition is ordered, in the sense that a distinguished point either lies on the disk bounded by the left or the right side of $\alpha_i$ (when oriented so as to run from $\infty$ to $z_i$.

            The {\em type} of $\gamma$ is defined to be this ordered partition.
        \end{definition}

        \begin{lemma}
            \label{lemma:transactiondisjoint}
            The stabilizer $G_0 \le \Gamma_\kappa$ of $A_0$ acts transitively on admissible arcs of a fixed type disjoint from $A_0$.
        \end{lemma}
        \begin{proof}
            It is easy to see that the pure braid group of $\C_\kappa\setminus A_0$ acts transitively on arcs of a given type. By \Cref{lemma:monodromycontainsstab}, this is a subgroup of $G_0 \le \Gamma_\kappa$.
        \end{proof}
        
        \begin{lemma}\label{lemma:monodromytransadjacent}
            Let $A' \in \bM_\kappa$ be an ARM adjacent to $A_0$. Then there is $g \in \Gamma_\kappa$ for which $g(A_0) = A'$.
        \end{lemma}

\begin{figure}[h!]
\labellist
\tiny
\pinlabel \textcolor{myblue}{$\alpha$} at 20 600
\pinlabel \textcolor{myblue}{$\alpha''$} at 320 75
\pinlabel (1) at 110 490
\pinlabel (2) at 260 490
\pinlabel (3) at 410 490
\pinlabel (4) at 110 250
\pinlabel (5) at 260 250
\pinlabel (6) at 410 250
\pinlabel (7) at 110 10
\pinlabel (8) at 260 10
\pinlabel (9) at 410 10
\endlabellist
\includegraphics[scale = 0.75]{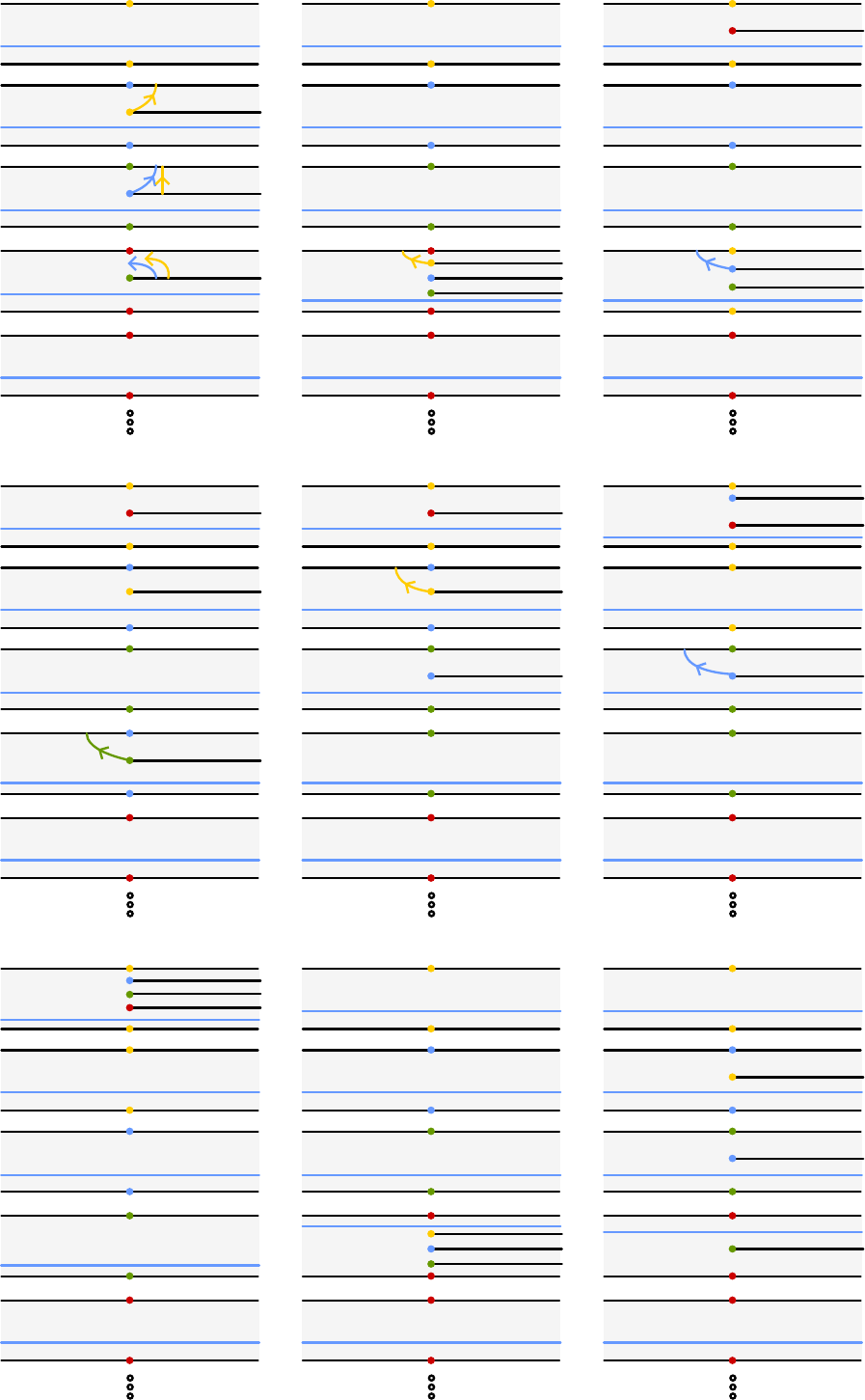}
\caption{A deformation based at $T_\sigma$ taking $A_\sigma$ to $A_\sigma''$.}
\label{figure:newarm}
\end{figure}
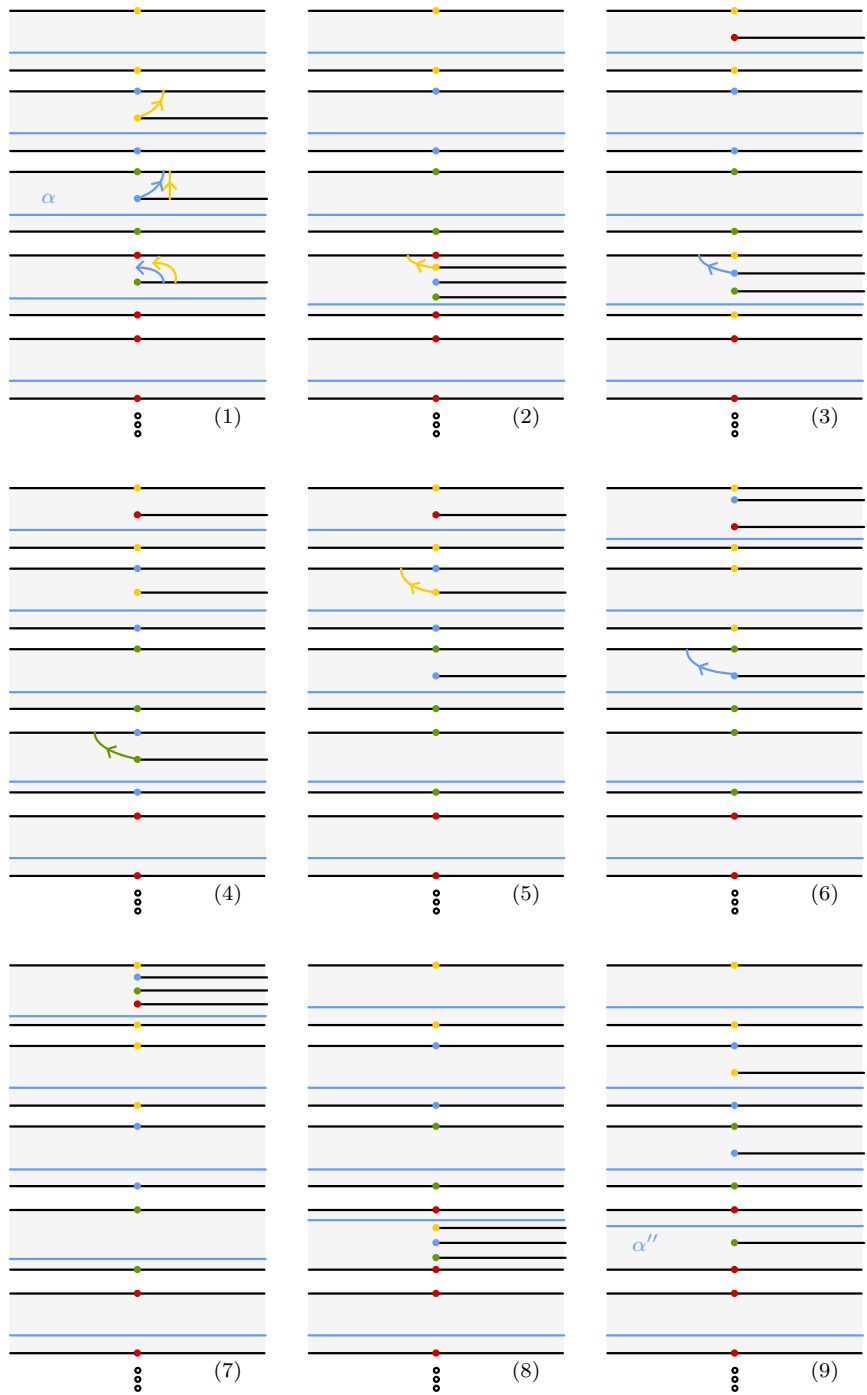
        
        \begin{proof}
            Let $\alpha_i'$ be the unique element of $A'$ not contained in $A_0$. Then $\alpha_i \cup \alpha_i'$ divides the set of critical points into two groups $C$ and $C'$, lying to the right (resp. left) of $\alpha_i$ when oriented as usual to run from $\infty$ to a root.
            
            Observe that the set of roots enclosed by $\alpha_i \cup \alpha_i'$ to the right of $\alpha_i$ is in fact determined by $C$: letting $k$ denote the sum of the orders of the points in $C$, the roots enclosed with $C$ consist of the $k$ roots $z_{i+1}, \dots, z_{i+k}$ lying immediately clockwise from $\alpha_i$ in the cyclic ordering of the roots specified by the ARM $A_0$ (cf. \Cref{remark:armordersroots}). Thus the type of $\alpha_i'$ relative to $A_0$ is determined only by the enclosed critical points $C$.

            Let $\sigma$ be an ordering of the critical points for which the the points of $C$ appear last (in any internal order). By \Cref{lemma:basepaths}, there is a path from $T_0$ marked with $A_0$ to $T_\sigma$ marked with $A_\sigma$. Let $A_\sigma'$ denote the ARM on $T_\sigma$ obtained by parallel transport of $A'$ along this path. 
            
            \Cref{figure:newarm} exhibits an element of the monodromy based at $T_\sigma$ taking $A_\sigma$ to some adjacent ARM $A_\sigma''$ of the same type as $A_\sigma'$; these differ only at the arcs marked $\alpha$ and $\alpha''$ in the first and last panels of the figure. In the illustrated example, the set $C$ consists of the critical points marked in green, blue, and yellow (the topmost three). For simplicity we illustrate the case where each of these has order $1$; the deformations in the general case are completely unchanged. Passing from (1) to (2), we push all but the bottom-most critical point in $C$ up as shown (if $C$ consists of only one element, this step is not performed). Denote the strip containing all these free prongs by $S$. Passing from panels (2) through (5), we work our way from top to bottom, pushing the topmost free prong in $S$ up through the top-left; this necessitates repeatedly recutting $S$ so that the topmost free prong becomes fixed. Passing from panels (5) to (7), we push the free prongs for the same subset of critical points in $C$ up through the top-left of their slits; this again requires a recutting. Passing from (7) to (8) we simply re-order the strips vertically and recut the top strip in (7). Finally, one passes from (8) to (9) by following the inverse of the path taken from (1) to (2).
            
            We have thus exhibited an element of the monodromy (based at $T_\sigma$) taking $A_\sigma$ to $A_\sigma''$ of the same type as $A_\sigma'$. Let $A''$ denote the ARM on $T_0$ obtained by parallel transport of $A_\sigma''$ back from $T_\sigma$ to $T_0$. Changing the basepoint back to $T_0$, we obtain a loop $g$ based at $T_0$ that takes $A_0$ to some $A''$ of the same type as $A'$. By \Cref{lemma:transactiondisjoint}, there is some element of $\Gamma_\kappa$ that fixes $A_0$ and sends $A''$ to $A'$. Composing these, we produce the required element $g \in \Gamma_\kappa$ with $g(A_0) = A'$.
        \end{proof}

        \begin{proof}[Proof of \Cref{theorem:main}]
            Let $\psi_T$ be the logarithmic relative winding number function on $\C_\kappa$ (as defined in \Cref{definition:lrwnf}), and let $f \in B_\kappa[\psi_T]$ be arbitrary. With $A_0$ continuing to denote the basepoint ARM on $T_0$ as in \Cref{construction:T0}, define $B = f(A_0)$. By \Cref{lemma:armconnected}, there is a sequence $A_0, A_1, \dots, A_m = B$ of adjacent vertices in the graph $\bM_\kappa$ of ARMs on $\C_\kappa$. 

            We will see how to construct $f' \in \Gamma_\kappa$ for which $f'(A_0) = B$, by inductively defining elements $f_i \in \Gamma_\kappa$ for which $f_i(A_0) = A_i$. By \Cref{lemma:monodromytransadjacent}, there exists $g_1 \in \Gamma_\kappa$ such that $g_1(A_0) = A_1$; define $f_1:=g_1$. Now define $A_{i+1}' = f_i^{-1}(A_{i+1})$, and note that by induction, $A_{i+1}'$ is adjacent to $f_i^{-1}(A_i) = A_0$. Again by \Cref{lemma:monodromytransadjacent}, there is $g_{i+1} \in \Gamma_\kappa$ such that $g_{i+1}(A_0) = A_{i+1}'$, and then set $f_{i+1} = f_i g_{i+1}$, and verify that $f_{i+1}(A_0) = f_i(A_{i+1}') = A_{i+1}$.

            Given such $f'$, note that the composition $f^{-1}f'$ fixes $A_0$ and hence $f^{-1} f' \in G_0 \le \Gamma_\kappa$ by \Cref{lemma:monodromycontainsstab}. It follows that $f \in \Gamma_\kappa$ as desired.
        \end{proof}

\subsection{Proof of \Cref{corollary:main}}\label{subsection:maincor} 

Here we see how to recover and improve the main theorem of \cite{stratbraid1}. We will be brief here and will freely refer back to \cite{stratbraid1} as required.

\cite[Section 4]{stratbraid1} establishes the theory of relative winding number functions on $\C$ marked only at the roots of a polynomial $f \in \Poly_n(\C)[\kappa]$. In \cite[Lemma 4.3]{stratbraid1}, it is shown that, taking $r = \gcd(k_1, \dots, k_p)$, there is a well-defined ``mod-$r$ winding number function'' $\bar \psi_T$ from the set of arcs connecting $\infty$ to a root, valued in $\Z/r\Z$, computed as the mod $r$-reduction of the associated logarithmic winding number function $\psi_T$ on $\C_\kappa$. To prove \Cref{corollary:main}, it suffices to show that the forgetful map $B_\kappa \to B_n$ (induced by forgetting the critical points) induces a surjection $B_\kappa[\psi_T] \to B_n[\bar \psi_T]$. 

To this end, let $\bar f \in B_n[\bar \psi_T]$ be arbitrary, and lift $\bar f$ to $f \in B_\kappa$. Let $A$ be an ARM on $\C_\kappa$. Then $f(A)$ is a root marking, and by hypothesis, each arc in $f(A)$ is isotopic to an admissible arc if allowed to slide over critical points. Via \Cref{lemma:slide}, each such crossing changes the winding number by a multiple of $r$.

For each root $z_i$, choose a system of arcs disjoint from $f(A)$ and connecting $z_i$ to the critical points $w_m$. By twist-linearity, the Dehn twist about a neighborhood of such an arc changes the winding number of the arc in $f(A)$ at $z_i$ by the corresponding order $k_m$, and leaves the winding numbers of each other arc in $f(A)$ unchanged. By performing some suitable set of twists, it is possible to successively alter each of the arcs in $f(A)$ so that the winding numbers become zero. By \Cref{lemma:framedcriterion}, the composition of $f$ with such a collection of twists lies in $B_\kappa[\psi_T]$, and induces the chosen $\bar f$ upon forgetting the critical points. \qed

    \bibliography{references}{}
	\bibliographystyle{alpha}

\end{document}